\documentclass[12pt]{amsart}
\usepackage{amssymb}
\newtheorem{theorem}{Theorem}[section]
\newtheorem{lemma}[theorem]{Lemma}
\newtheorem{prop}[theorem]{Proposition}
\newtheorem{cor}[theorem]{Corollary}
\theoremstyle{definition}
\newtheorem{definition}[theorem]{Definition}

\theoremstyle{remark}
\newtheorem{remark}[theorem]{Remark}
\newtheorem{problem}[theorem]{Problem}
\numberwithin{equation}{section}

\newcommand\B{\mathbb{B}}
\newcommand\C{\mathbb{C}}
\newcommand\D{\mathbb{D}}

\newcommand\R{\mathbb{R}}
\newcommand\N{\mathbb{N}}
\newcommand\T{\mathbb{T}}
\newcommand\cA{\mathcal{A}}
\newcommand\U{\mathcal{U}}
\newcommand\cD{\mathcal{D}}

\newcommand\cH{\mathcal{H}}
\newcommand\cK{\mathcal{K}}
\newcommand\cI{\mathcal{I}}
\newcommand\cL{\mathcal{L}}

\newcommand\cS{\mathcal{S}}
\newcommand\cT{\mathcal{T}}

\newcommand\RE{\mathrm{Re}\,}
\newcommand\IM{\mathrm{Im}\,}
\newcommand\EXP{\mathrm{Exp}\,}
\newcommand\Ran{\mathrm{Ran}\,}
\newcommand\Ker{\mathrm{Ker}\,}
\newcommand\spa{\mathrm{span}\,}
\newcommand\RAU{R\cA\U}
\newcommand\IAU{I\cA\U}
\newcommand\HD{\cH\cD}
\newcommand\rH{\mathbb{H}_r}
\newcommand\Hi{L^2(0,\infty)}
\newcommand\HiR{\Hi_\R} 
\newcommand\LR[1]{L^2(0,#1)_\R}
\newcommand\Ha{H^2(\rH)}
\newcommand\inpr[2]{\langle{#1,#2}\rangle}

\newcommand\Lt[1]{\cL{[#1]}}

\begin{document}
\title{Generalized CCR flows}

\author{MASAKI IZUMI}
\address{Department of Mathematics,\\ 
Kyoto University, Kyoto, Japan.}
\email{izumi@math.kyoto-u.ac.jp}

\author{R. SRINIVASAN}
\address{Chennai Mathematical Institute\\ 
Siruseri 603103, India.}
\email{vasanth@cmi.ac.in}



\thanks{Work supported by JSPS}
\begin{abstract} 
We introduce a new construction of $E_0$-semigroups, called generalized CCR flows, with two kinds of 
descriptions: those arising from sum systems and those arising from pairs of $C_0$-semigroups. 
We get a new necessary and sufficient condition for them to be of type III, when the associated sum system 
is of finite index. 
Using this criterion, we construct examples of type III $E_0$-semigroups, which can not be distinguished 
from $E_0$-semigroups of type I by the invariants introduced by Boris Tsirelson. 
Finally, by considering the local von Neumann algebras, and by associating a type III 
factor to a given type III $E_0$-semigroup, we show that there exist uncountably many type III 
$E_0$-semigroups in this family, 
which are mutually non-cocycle conjugate.
\end{abstract}

\maketitle

%
%\begin{center}
%{\large{\bf Generalized CCR flows}}\\

%\vskip 1.5em

%MASAKI IZUMI,\\
%Department of Mathematics,\\ 
%Kyoto University, Kyoto, Japan.\\
%izumi@math.kyoto-u.ac.jp\\

%\bigskip

%R. SRINIVASAN,\\
%Chennai Mathematical Institute\\ 
%Siruseri 603103, India.\\
%vasanth@cmi.ac.in

%\bigskip

%\today
%\end{center}

%
%\vskip1.5em

%\begin{center}
%{\bf Abstract}
%\end{center}

% We generalize the CCR flows and show that it is same as the $E_0$-semigroups arising from sum systems. We get a new necessary and sufficient condition for them to be of type III, when the associated sum system is of index $1$. Then we apply the results obtained in\cite{I} to produce new examples of $E_0$-semigroups belonging to type III, which can not  be distinguished from $E_0$-semigroups of type I by the invariants introduced by Boris Tsirelson. Finally, by considering the local von Neumann algebras, and by associating a type III factor to a given type III $E_0$-semigroup, we show that there exists uncountably many type III $E_0$-semigroups in this family, which are mutually non-isomorphic.

%
%\bigskip
%

%\noindent {\bf AMS subject classification:} 46L55, 30D50, 32A10, 81S05.

%\noindent {\bf Key words:} $E_0$-semigroups, CCR Flows, Type III, Unilateral Shift, Hilbert-Schmidt. 

%\bigskip 

\section{Introduction}\label{intro} 
An $E_0$-semigroup is a weakly continuous semigroup of unital $*$-endomorphisms on $\B(H)$, 
the algebra of all bounded operators on a separable Hilbert space $H$. 
$E_0$-semigroups are classified into three broad categories, namely type I, II, and III, 
depending upon the existence of intertwining semigroups called units.  
William Arveson completely classified the $E_0$-semigroups of type I, by showing that the CCR flows exhaust 
all the type I $E_0$-semigroups, up to the identification of cocycle conjugacy (\cite{Arv}). 
But the theory of $E_0$-semigroups belonging to type II and type III remained mysterious for quite some time. 
There is no hope of completely classifying the whole class of $E_0$-semigroups even now, basically due to 
the presence of type II and type III examples. 

For quite sometime there were essentially only one example for each type II and type III $E_0$-semigroups, 
due to R. T. Powers (\cite{Po1}, \cite{Po2}). 
In this context Boris Tsirelson produced uncountable families of both type II and type III $E_0$-semigroups by using measure type 
spaces arising from several models in probability theory (\cite{T1}). 
It is equivalent to study the product systems of Hilbert spaces, a complete invariant introduced by Arveson, 
in order to understand the associated $E_0$-semigroups. 
Tsirelson basically produced uncountable families of both type II and type III product systems of Hilbert spaces.

Tsirelson's construction of type III product systems uses off white noises, which are Gaussian generalized 
(i.e. distribution valued) processes with a slight correlation between ``past and future". 
After discussing Tsirelson's results, Arveson concludes his book (\cite{Arv}) by saying, 
`It is clear that we have not achieved a satisfactory understanding of the existence of `logarithms' in the category of 
product systems.' 
This was clarified for Tsirelson's type III examples in \cite{pdct} from the viewpoint of operator algebras. 
Inspired by the results of Tsirelson, a purely operator algebraic construction of Tsirelson's type III examples was provided 
in \cite{pdct}. 
They were called as `product systems arising from sum systems'. 
In particular, a dichotomy result about types was proved in \cite{pdct}, namely, it was proved that the product system arising from a 
divisible sum system is either of type I or of type III. 
A sum system is said to be divisible if it has sufficiently many real and imaginary addits (called additive 
cocycles in \cite{pdct}).

On the other hand, motivated by Tsirelson's construction of type III examples, a class of $C_0$-semigroups acting on $L^2(0,\infty)$, 
which are Hilbert-Schmidt perturbations of the unilateral shift semigroup, was investigated in \cite{I}. 
Description of such semigroups in terms of analytic functions on the right-half plane was given and several examples were constructed. 

In this paper we discuss the consequences of these developments. 
First we describe the $E_0$-semigroups associated with the above mentioned product systems arising from 
sum systems. This would generalize the simplest kind of $E_0$-semigroups called CCR flows. 
These generalized CCR flows are given by a pair of $C_0$-semigroups. 
By studying the product system we get a new necessary and sufficient condition for the $E_0$-semigroup 
to be of type III, when the associated sum system is of finite index. 
This criterion is much more powerful than the sufficient condition already proved in \cite{pdct}. 

Using the results proved in \cite{I}, we compute the additive cocycles for the pairs of the shift semigroup and its 
perturbations, and show that the sum systems associated with them are always divisible. 
This class of sum systems include those coming from off white noises of Tsirelson. 
Then we concentrate on a special subclass of $C_0$-semigroups, which give rise to new type III $E_0$-semigroups. 
These new examples cannot be distinguished from type I $E_0$-semigroups by the 
invariants introduced by Tsirelson, and later discussed in \cite{pdct}. 
Let us give an intuitive explanation of this phenomenon in terms of Tsirelson's off white noise picture here.   
Although our examples also come from Tsirelson's off white noises, 
spectral density functions for them tend to 1 at infinity.  
This means that our off white noises are so close to white noise that Tsirelson's invariant can not work 
for them. 

Finally, we associate a type III factor as an invariant to each of these type III $E_0$-semigroups, 
and using that we prove that there are uncountably many examples in this family which are not cocycle 
conjugate to each other. 
Toeplitz operator plays an essential role throughout these discussions.  

We end this section by reviewing some of the very basic definitions about $E_0$-semigroups. 
For the definitions of notions related to $E_0$-semigroups, such as cocycle conjugacy, index etc.,  
we refer to \cite{Arv}. 
A unit for an $E_0$-semigroup $\{\alpha_t\}$  acting on $\B(H)$ is a strongly continuous semigroup of 
bounded operators $\{T_t\}$, which intertwines $\alpha$ and the identity, that is 
$$\alpha_t(X)T_t = T_tX, ~~~~~\forall ~A \in \B(H), ~t\geq 0.$$

Product systems, a complete invariant for cocycle conjugacy introduced by Arveson, are demonstrated to be a very useful tool in 
studying $E_0$-semigroups.  
For an $E_0$-semigroup $\{\alpha_t; t \geq 0\}$ these are Hilbert spaces of the intertwining operators 
$\{H_t; t\geq 0 \}$, defined by  $$H_t =\{T \in \B(H); \alpha_t(X) T = TX, ~ \forall ~X \in
\B(H)\}$$ with inner product $\langle T, S \rangle 1_H= S^*T$ (see
\cite{Arv}). 
Every notion related to $E_0$-semigroups can be translated into the framework of the associated product systems. 
We provide here a  slightly different (but equivalent) definition  than originally defined in \cite{Ar} and \cite{Arv}. 
The difference is  in the measurability axiom (and we do not need a separate non-triviality axiom), which is due to Volkmar Liebcher. 
He has also proved in \cite{Vol} that any two measurable structures on a given algebraic product system give rise to isomorphic product 
systems, and as a consequence we get that two product systems are isomorphic if they are
algebraically isomorphic.  So we need not consider measurable structures while dealing with isomorphism of product systems. 

\begin{definition}\label{productsystem}
A product system of Hilbert spaces is an one parameter family
of separable complex Hilbert spaces $\{H_t;t \geq 0 \}$, together with unitary operators 
$$U_{s,t} : H_s \otimes H_t ~\mapsto H_{s+t}~ \mbox{for}~ s, t \in (0,\infty),$$
satisfying the following two axioms of associativity and measurability.

\medskip
\noindent(i) (Associativity) For any $s_1, s_2, s_3 \in (0,\infty)$
$$U_{s_1, s_2 + s_3}( 1_{H_{s_1}} \otimes U_{s_2 ,
s_3})=  U_{s_1+ s_2 , s_3}( U_{s_1 ,
s_2} \otimes 1_{H_{s_3}}).$$

\noindent (ii) (Measurability) There exists a countable set $H^0$ of
sections $ R\ni t \rightarrow h_t \in H_t$ such that $ t  \mapsto 
\langle h_t, h_t^\prime\rangle$ is measurable for any two $h, h^\prime \in H^0$, and the set $\{h_t; h \in H^0\}$ is total 
in $H_t$, for each $ t \in (0,\infty)$. 
Further it is also assumed that the map $(s,t) \mapsto \langle U_{s,t}(h_s \otimes h_t), h^\prime_{s+t} \rangle$ is measurable 
for any two $h , h^\prime \in H^0$.
\end{definition}

\medskip

\begin{definition}
Two product systems $(\{H_t\}, \{U_{s,t}\})$ and $(\{H^\prime_t\}, \{U^\prime_{s,t}\})$
are said to be isomorphic if there exists a unitary operator $V_t:H_t \mapsto
H_t^\prime$, for each
$t \in (0,\infty)$ satisfying $$V_{s+t}U_{s,t}= U_{s,t}^\prime (V_s \otimes V_t).$$
\end{definition}

\medskip

\begin{definition}\label{unit}
A unit for product system is a non-zero section $\{u_t;t \geq 0 \}$, such that the map $t \mapsto \langle u_t, h_t\rangle$ is
measurable for any $h \in H^0$ and 
$$U_{s,t}(u_s \otimes u_t)= u_{s+t}, ~  \forall s,t \in (0,\infty).$$
\end{definition}

\medskip

A product system ($E_0$-semigroup) is said to be of type I, if units exists for the product system and 
they generate the product system, i.e. for any fixed $t \in (0,\infty)$, the set $$\{u^1_{t_1}u^2_{t_2}  
\cdots u^n_{t_n}; \sum_{i=1}^n t_i = t, u^i \in \U\},$$ 
is a total set in $H_t$, where $\U$ is the set of all units and the product is defined as the image of
$u^1_{t_1}\otimes u^2_{t_2}  \cdots \otimes u^n_{t_n}$ in $H_t$, under the canonical unitary given by 
the associativity axiom. 
It is of type II if units exist but they do not generate the product system. 
We say a product system to be of type III or unitless if there does not exist any unit
for the product system.

\bigskip

%%%%%%%%%%%%%%%%%%%%%%%%%%%%%%%%%%%%%%%%%%%%%%%%%%%%%%%%%%%%%%%%%%%%%%%%%%%%%%%%%%%%%%%%%%%%%%%%%%%%%%%%%%%%%%
%%%%%%%%%%%%%%%%%%%%%%%%%%%%%%%%%%%%%%%%%%%%%%%%%%%%%%%%%%%%%%%%%%%%%%%%%%%%%%%%%%%%%%%%%%%%%%%%%%%%%%%%%%%%%%
%%%%%%%%%%%%%%%%%%%%%%%%%%%%%%%%%%%%%%%%%%%%%%%%%%%%%%%%%%%%%%%%%%%%%%%%%%%%%%%%%%%%%%%%%%%%%%%%%%%%%%%%%%%%%%
\section{Preliminaries and notation}\label{pre}

In this section we fix the notation used in this paper, recall some of the results proved earlier and make a few definitions. 
For an operator $A$, we denote the range of $A$ by $\Ran(A)$ and the kernel of $A$ by $\Ker(A)$.  
We denote the identity operator on a Hilbert space $H$ by $1_H$ (or simply by 1 if no confusion arises).

For a complex Hilbert space $K$, we denote by $\Gamma(K)$ the symmetric Fock space associated with $K$, 
and by $\Phi$ the vacuum vector in $\Gamma(K)$ (see \cite{KRP}). 
For any $x \in K$, the exponential vector of $x$ is defined by $$e(x)= \oplus_{n=0}^\infty  
\frac{x^{\otimes^n}}{\sqrt{n!}},$$ where $x^0=\Phi$.  Then the set of all exponential vectors 
$\{e(x): x \in K\}$ is a linearly independent total set in $\Gamma(K)$. 
For a unitary $U\in \B(K)$, we denote by $\EXP(U)$ the unitary in $\B(\Gamma(K))$ 
given by $\EXP(U)e(x)=e(Ux)$. 
The Weyl operator, corresponding to an element $x \in K$, is defined by, 
$$W(x)(e(y))= e^{-\frac{\|x\|^2}{2}- \langle y, x
\rangle} e(y+x),$$ 
and $W(x)$ extends to a unitary operator on $\Gamma(K)$. 
The $*$-algebra generated by $\{W(x)\}_{x\in K}$ is called the Weyl algebra for $K$. 

For a real Hilbert space $G$, we denote  the complexification of $G$ by $G^{\C}$. 
For two Hilbert spaces $G_1$, $G_2$, define 
$$\cS(G_1,G_2)=\{A\in \B(G_1, G_2);  A~\mbox{invertible and} ~I-(A^*A)^{\frac{1}{2}}~\mbox{Hilbert-Schmidt}\}.$$ 
In the above definition and elsewhere, by invertibility we mean that the inverse is also bounded. 
The set $\cS(., .)$ is well behaved with respect to taking inverses,
adjoints, products, and restrictions (see \cite{pdct}). 
For two real Hilbert spaces $G_1, G_2$ and
$A \in \cS(G_1,G_2 )$, define
a real liner operator $S_A:G_1^{\C} \rightarrow
G_2^{\C}$ by
$S_A(u+iv)=Au+i(A^{-1})^*v$ for $u,v \in
G_1$. Then $S_A$ is a symplectic
isomorphism between $G_1^\C$ and
$G_2^\C$
(i.e. $S_A$ is a real linear, bounded,  invertible map satisfying $\IM(\langle S_Ax,S_Ay\rangle
)=\IM\langle x,y\rangle $ for all $x,y
\in G_1^\C$, see
\cite[page 162]{KRP}). In general $S_A$ is not complex linear, unless $A$ is unitary.

The following theorem, a generalization of Shales theorem, is used to 
construct the product system from a given sum system in \cite{pdct}.

\begin{theorem}\label{shales} {\rm (i)} Let $G_1,
G_2$ be real Hilbert spaces and $A \in \cS(G_1,G_2)$ ,
then there exists a unique unitary operator
$\Gamma(A):\Gamma(G_1^\C)
\rightarrow
\Gamma(G_2^\C)$ such that

\begin{eqnarray}\label{SHL}
\Gamma(A)W(u){\Gamma(A)}^* & = & W(S_Au) ~~~ \forall ~u \in G_1^{\C}\label{SHL1}\\
\langle \Gamma(A) \Phi_1, \Phi_2\rangle & \in & \R^{+}\label{SHL2}
\end{eqnarray}
where $\Phi_1$ and $\Phi_2$ are the vacuum
vectors in $\Gamma(G_1^\C)$ and
$\Gamma(G_2^\C)$
respectively.

Conversely, for a given bounded operator $A:G_1 
\mapsto G_2$, if
there exists a unitary operator 
$\Gamma(A):\Gamma(G_1^\C)
\rightarrow
\Gamma(G_2^\C)$, 
satisfying, $\Gamma(A)W(u){\Gamma(A)}^* =
W(S_Au),~\forall  ~u \in G_1^{\C}$ then $A \in \cS(G_1, G_2)$.

\medskip \noindent
{\rm (ii)} Suppose $G_1, G_2, G_3$ be three real Hilbert
spaces, and
$A \in \cS(G_1, G_2),~B \in \cS(G_2, G_3)$,  then
\begin{eqnarray}\label{shlrep}
\Gamma(A^{-1}) & = & {\Gamma(A)}^*\\\label{shlrep1}
\Gamma(BA) & = & \Gamma(B)\Gamma(A)
\end{eqnarray}
\end{theorem}

\begin{proof}
We only have to prove the converse part of (i). Remaining parts are 
already proved in \cite{pdct}. For $A \in \B(G_1, G_2)$, consider the 
polar decomposition of 
$A= UA_0$, where $A_0 \in \B(G_1)$, and $U \in \B(G_1, G_2)$ a 
unitary operator. Suppose there exists a unitary operator
$\Gamma(A):\Gamma(G_1^{\C})
\rightarrow
\Gamma(G_2^{\C})$,
satisfying $\Gamma(A)W(u){\Gamma(A)}^* =
W(S_Au),$ then 
$\Gamma(A_0)=\EXP(U^*)\Gamma(A)$ will satisfy, 
$\Gamma(A_0)W(u){\Gamma(A_0)}^* =
W(S_{A_0}u).$ The original version of Shales theorem will imply  
that 
$A_0 \in \cS(G_1, G_1)$ (see \cite[page 169]{KRP}). 
So we can conclude that $A \in \cS(G_1, G_2)$.
\end{proof}

\medskip

Next we define the notion of a sum system. 

\begin{definition}\label{sumsystem}
A sum system is a two parameter family of real Hilbert spaces $\{G_{s,t}\}$ for $0 <s < t \leq \infty$, satisfying 
$G_{s,t} \subset G_{s^{\prime}, t^{\prime}}$ if the interval $(s,t)$ is contained in the interval $(s^{\prime}, t^{\prime})$, 
together with a one parameter semigroup $\{S_t\}$, of bounded linear operators on $G_{(0, \infty)}$ for $t \in (0,\infty)$ such that
\begin{itemize}
\item [(i)] $S_s|_{G_{0,t}} \in \cS(G_{0,t} , G_{s, s+t})$ $\forall~t \in (0,\infty],~s \in [0,\infty)$. \medskip
\item [(ii)] If $A_{s,t}:G_{0,s}\oplus G_{s,s+t}\mapsto G_{0,s+t}$, is  the map
$A_{s,t}(x\oplus y)=x+y$, for $x \in G_{0,s}, y \in G_{s,s+t}$, then 
$A_{s,t} \in \cS(G_{0,s}\oplus G_{s,s+t}, G_{0,s+t})$, $\forall ~s,t \in
(0,\infty)$. \medskip
\item [(iii)] The semigroup $\{S_t\}$ is strongly continuous. 
\end{itemize}

We say that two sum systems $(\{G_{a,b}\},\{S_t\})$ and $(\{G'_{a,b}\},\{S'_t\})$ are isomorphic 
if there exists a family of operator $U_t\in \cS(G_{0,t},G'_{0,t})$ preserving 
every structure of the sum systems, (i.e) $\{U_t\}$ satisfies $$ A^\prime_{s,t}(U_s \oplus S^\prime_s|_{G_{0,t}^\prime} U_t) = U_{s+t}A_{s,t}(1_{G_{0,s}}\oplus S_s|_{G_{0,t}}),$$ where $A_{s,t}^\prime$ is defined  for $(\{G'_{a,b}\},\{S'_t\})$ similar to $A_{s,t}$ above.
\end{definition}

The above definition is slightly stronger than the one given in \cite{pdct}. 
Namely, we require here that $\{S_t\}$ is a $C_0$-semigroup acting on the global Hilbert space $G_{0,\infty}$  and axiom (ii) holds for $t=\infty$ also,
though they were not assumed in \cite{pdct}. 

Given a sum system $(\{G_{s,t}\},\{S_t\})$, we define Hilbert spaces $H_t=\Gamma(G_{0,t}^\C)$, and  unitary operators
$U_{s,t}:H_s\otimes H_t \mapsto H_{s+t},$ by
$U_{s,t}=\Gamma(A_{s,t})(1_{H_s}\otimes \Gamma(S_s|_{G_{0,t}})).$ It is
proved in \cite{pdct} that $(\{H_t\}, \{U_{s,t}\})$ forms a product system. 
Isomorphic sum systems give rise to isomorphic product systems. 

Let $K$ be a complex Hilbert space and let $\{S_t\}$ be the shift semigroup of $L^2((0,\infty),K)$ defined by 
\begin{eqnarray*}(S_tf)(s) & = & 0, \quad s<t,\\
& = & f(s-t), \quad s \geq t,
\end{eqnarray*} 
for $f \in L^2((0,\infty),K).$
The CCR flow of index $\dim K$ is the $E_0$-semigroup $\alpha$ acting on $\B(\Gamma(L^2((0,\infty),K)))$ 
defined by $\alpha_t(W(f))=W(S_tf)$. 

To generalize the CCR flows, we ask the following question. 
Let $G$ be a real
Hilbert space and $H = \Gamma(G^\C)$.
Suppose we have two semigroups of linear operators, $S_t, ~T_t:G \mapsto G$ for
$t\geq 0$. 
Consider the association
$$\alpha_t(W(x))
\mapsto W(S_tx),~ \alpha_t(W(iy)) \mapsto W(iT_ty),  ~ x, y \in G.$$ 
When can we extend  this map to an $E_0$-semigroup on $\B(H)$? 
The continuity and the semigroup property of $\{\alpha_t\}$ will
immediately imply that both $\{S_t\}$ and $\{T_t\}$ have to be 
$C_0$-semigroups. 
Also, $\alpha_t$ being an endomorphism satisfies $$\alpha_t(W(u)W(v))=\alpha_t(W(u))\alpha_t(W(v)), 
u,v \in G^\C.$$ Comparing both sides, 
using the 
canonical commutation relation, we get $$\langle S_tx, T_t 
y\rangle=\langle x, y\rangle ~~\forall~ x,y \in G$$ which is same as saying 
$T_t^*S_t=1.$ 
Assume that these conditions are satisfied. 
Then for $\alpha_t$ to extend as an endomorphism of $\B(H)$, it is necessary and sufficient that 
$\alpha_t$, as a representation of the Weyl algebra, is quasi-equivalent to the defining (vacuum) representation.  

The following lemma, probably well-known among specialists, gives a complete answer to the above question.  
For convenience of the reader, we include a proof here. 

\begin{lemma} \label{quasi-equivalence} Let $G$ be a real Hilbert space and let $S$ and $T$ be 
operators in $\B(G)$ satisfying $T^*S=1$. 
Then the representation $\pi$ of the Weyl algebra for $G^\C$ given by 
$$\pi(W(x+iy))=W(Sx+iTy),\quad x,y\in G$$
is quasi-equivalent to the defining representation if and only if $S-T$ is a 
Hilbert-Schmidt class operator. 
\end{lemma}

\begin{proof} Thanks to the relation $T^*S=1$, the above $\pi$ actually gives a representation of the Weyl algebra. 
Since $T^*$ (resp. $S^*$) is a left inverse of $S$ (resp. $T$), the range of $S$ (resp. $T$) is closed 
and $|S|$ (resp. $|T|$) is invertible.

We first claim that $\pi$ is a factor representation. 
Let $K_1=\Ran(S)$ and $K_2=\Ran(T)$, which are closed subspaces of $G$. 
Then the relation $T^*S=1$ implies that we have $K_1\cap K_2^\perp=K_2\cap K_1^\perp=\{0\}$, and so 
$$M:=\{\pi(W(z));\; z\in G^\C\}''=\{W(x+iy);\; x\in K_1,\; y\in K_2\}''$$  
is a factor thanks to \cite[Theorem 1]{Ara1}. 
The claim implies that $\pi$ is quasi-equivalent to the restriction of $\pi$ to $\overline{M\Phi}$, 
which is unitarily  equivalent to the GNS representation of the quasi-free state 
$\omega'(x)=\inpr{\pi(x)\Phi}{\Phi}$. 
Therefore to prove the statement, it suffices to show that the GNS representations of the vacuum state 
$\omega$ and the quasi-free state $\omega'$ are quasi-equivalent if and only if $S-T$ is a Hilbert-Schmidt operator. 
For that we can apply well-known criteria in \cite{Ara3}, \cite{vD}. 
Note that $\omega$ and $\omega'$ are given by 
$$\omega(W(x+iy))=e^{-(\|x\|^2+\|y\|^2)/2},\quad, x,y\in G,$$
$$\omega'(W(x+iy))=e^{-(\inpr{|S|^2x}{x}+\inpr{|T|^2y}{y})/2},\quad x,y\in G. $$

Since we deal with Weyl operators rather than creation and annihilation operators, we use the criterion 
in \cite{vD}, and we first recall the notation there. 
Let $\sigma(z_1,z_2)=-\IM\inpr{z_1}{z_2}$, which is a symplectic form of $G^\C$ as a real vector space. 
Then we have the Weyl relation $W(z_1)W(z_2)=e^{-i\sigma(z_1,z_2)}W(z_1+z_2)$. 
Let $A(x+iy)=y-ix$ and $B(x+iy)=|T|^2y-i|S|^2x$ and let 
$s_A(z_1,z_2)=\sigma(Az_1,z_2)$, $s_B(z_1,z_2)=\sigma(Bz_1,z_2)$. 
Then the two states are given by $\omega(W(z))=e^{-s_A(z,z)/2}$ and $\omega'(W(z))=e^{-s_B(z,z)/2}$. 
To see that $\omega'$ is a quasi-free state, we have to verify that $B^*B-1$ is positive as an operator  
acting on the real Hilbert space $G^\C$ equipped with the inner product $s_B$. 
Since $B^*=-B$, all we have to show is $|S|^2|T|^2|S|^2\geq |S|^2$ and $|T^2||S|^2|T^2|\geq |T|^2$, or equivalently 
just $|S||T|^2|S|\geq 1$.  
Indeed, let $S=U|S|$ and $T=V|T|$ be the polar decomposition of $S$ and $T$ respectively. 
Then thanks to $T^*S=1$, we have $V^*U=|T|^{-1}|S|^{-1}$, and so we get $1\geq V^*UU^*V=|T|^{-1}|S|^{-2}|T|^{-1}$, 
which shows  $|S||T|^2|S|\geq 1$. 

Let $Q_A=A+A\sqrt{1+A^{-2}}=A$ and $Q_B=B+B\sqrt{1+B^{-2}}$, where $\sqrt{1+B^{-2}}$ is a positive operator 
acting on $(G^\C,s_B)$ satisfying $\sqrt{1+B^{-2}}^2=1+B^{-2}$. 
Then the two GNS representations are quasi-equivalent if and only if $Q_A^{-1}Q_B-1$ is a Hilbert-Schmidt 
operator \cite[page 190]{vD}. 
In the rest of the proof, the symbol $\equiv$ means equality up to a Hilbert-Schmidt class operator. 
Straightforward computation yields 
\begin{eqnarray*}
Q_A^{-1}Q_B(x+iy)&=&
|S|(1+\sqrt{1-|S|^{-1}|T|^{-2}|S|^{-1}})|S|x\\
&+&i|T|(1+\sqrt{1-|T|^{-1}|S|^{-2}|T|^{-1}})|T|y,
\end{eqnarray*}
and so  $Q_A^{-1}Q_B-1$ is a Hilbert-Schmidt operator if and only if the 
following two relations hold: 
\begin{equation}\label{q.e.1}
\sqrt{1-|S|^{-1}|T|^{-2}|S|^{-1}}\equiv |S|^{-2}-1,
\end{equation}
\begin{equation}\label{q.e.2}
\sqrt{1-|T|^{-1}|S|^{-2}|T|^{-1}}\equiv |T|^{-2}-1.
\end{equation}

Assume first that the equations (\ref{q.e.1}) and (\ref{q.e.2}) hold. 
Then (\ref{q.e.1}) implies 
$$1-|S|^{-1}|T|^{-2}|S|^{-1}\equiv 1-2|S|^{-2}+|S|^{-4},$$
which is equivalent to $|S|^{-2}+|T|^{-2}\equiv 2$. 
Therefore (\ref{q.e.1}), (\ref{q.e.2}), and this imply 
$$\sqrt{1-|S|^{-1}|T|^{-2}|S|^{-1}}+\sqrt{1-|T|^{-1}|S|^{-2}|T|^{-1}}\equiv 0.$$
Since the left-hand side is a positive operator, each term must be a Hilbert-Schmidt operator. 
This means that the both sides of (\ref{q.e.1}) and (\ref{q.e.2}) are Hilbert-Schmidt operators. 
In particular, we have $|S|^{-2}\equiv |T|^{-2}\equiv 1$ and the two operators $1-|S|^{-1}|T|^{-2}|S|^{-1}$ 
and $1-|T|^{-1}|S|^{-2}|T|^{-1}$ are trace class operators. 
Since this shows that $|S|^2-|T|^{-2}$ is a trace class operator, the operator   
$$(S-T)^*(S-T)=|S^2|+|T|^2-2=|S|^2-|T|^{-2}+(|T|^{-2}-1)^2|T|^2 $$ 
is a trace class operator. 
This shows that $S-T$ is a Hilbert-Schmidt operator. 

Assume conversely that $S-T$ is a Hilbert-Schmidt operator now. 
Then $0\equiv S^*(S-T)=|S|^2-1$, $0\equiv T^*(T-S)=|T|^2-1$ and 
$$(S-T)^*(S-T)=|S|^2+|T|^2-2$$ is a trace class operator. 
Therefore a similar computation as above shows that (\ref{q.e.1}) and (\ref{q.e.2}) hold. 
\end{proof}

\medskip

The above lemma allows us to introduce the notion of a generalized CCR flow. 

\begin{definition}\label{perturb} Let $\{S_t\}$ and $\{T_t\}$ be $C_0$-semigroups acting on a real Hilbert space $G$. 
We say that $\{T_t\}$ is a perturbation of $\{S_t\}$, if they satisfy,
\begin{enumerate}
\item[(i)] ${T_t}^*S_t =1.$
\item[(ii)] $S_t -T_t$ is a Hilbert-Schmidt operator.
\end{enumerate}
\medskip
Given a perturbation $\{T_t\}$ of $\{S_t\}$, we say that the $E_0$-semigroup $\{\alpha_t\}$ acting on 
$\B(\Gamma(G^\C))$ given by 
$$\alpha_t(W(x+iy))=W(S_tx+iT_ty),\quad x,y\in G$$
is a generalized CCR flow associated with the pair $\{S_t\}$ and $\{T_t\}$. 
\end{definition}
\medskip
In fact we can also show from Proposition \ref{StTt} (see Section \ref{genccr}) that a pair of $(\{S_t\}, \{T_t\})$ 
gives rise to a generalized CCR flow if $T_t$ is a perturbation of $S_t$. 
%but to arrive at the converse statement we need the above Lemma \ref{quasi-equivalence}.

\medskip

From Section \ref{cocycles} onwards we will follow the notations used in \cite{I}. 
We will denote by $\{S_t\}$ the shift semigroup of $L^2(0,\infty)$. 
For $f, g  \in L^2(0,\infty)$, we denote $$(f, g)=\int_0^\infty f(x)g(x)dx,$$
while $\inpr{f}{g}$ denotes the usual complex inner product.  
Let $\rH$ be the right-half plane $\{z \in \C: \RE{x} >0 \}$. 
For $z \in \rH$ we set $e_z(x)=e^{-zx}$. 
We denote by $L^1_{\mathrm{loc}}[0,\infty)$ the set of all measurable functions on $[0,\infty)$ which are integrable on 
every compact subsets of $[0,\infty)$. 
For $f \in L^1_{\mathrm{loc}}[0,\infty)$ and $a >1$, such that the function $e^{-ax}f(x) \in L^1(0,\infty)$, 
we denote by $\Lt f(z)$ the Laplace transform $$\Lt f(z)=\int_0^\infty f(x) e^{-zx}dx, ~~ \RE{z} >a.$$ 

Let $\HD$ be the set of holomorphic functions $M(z)$ on the right-half plane $\rH$ such that $M(z)/(1+z)$ belongs to 
the Hardy space $\Ha$ and $M(z)$ does not belong to $\Ha$. 
Then we can associate a differential operator $A_M$ through the procedure described in \cite[Sections 2 and 3]{I}. 
Namely, let $q\in \Hi$ such that $\Lt q(z)=M(z)/(1+z)$. 
Then $A_M$ is the differential operator $A_Mf(x)=-f'(x)$ whose domain $D(A_M)$ is the set of locally absolutely 
continuous functions $f\in \Hi$ such that $f'\in \Hi$ and 
$(f-f',q)=0$. 

We denote by $\HD_b$ the set of $M\in \HD$ such that $A_M$ generates a $C_0$-semigroup, and   
by $\HD_2$ the set of $M\in \HD_b$ such that $e^{tA_M}-S_t$ is a Hilbert-Schmidt operator 
for all $t\geq 0$, where $\{S_t\}$ is the shift semigroup. 
The semigroup $\{T_t=e^{tA_M}\}$ will satisfy the relation $T_t^*S_t=1$ for all $t \geq 0$. 
Conversely for any given $C_0$-semigroup $T_t$ satisfying $T_t^*S_t=1$, its generator $A$ can be described by a $M(z) \in \HD_b$ as above and 
$M(z)$ is unique up to a non-zero scalar multiple. 

On the other hand suppose $\{T_t=S_t + K_t\}$ be a $C_0$-semigroup on $L^2(0,\infty)$ such that $\{T_t\}$ 
is a perturbation of the shift semigroup $\{S_t\}$. 
Then it is proved in \cite[Section 4]{I} that there exists a measurable function $k(x,y)$ defined on $(0,\infty)^2$ and $a>0$, satisfying  \begin{equation}\label{k}k(x+t,y)  =  k(x, y+t) +\int_0^tk(x,s)k(t-s,y)ds,~~\forall ~t\geq 0,~\mbox{a.e}~(x,y)\in (0,\infty)^2\end{equation}
\begin{equation}\label{k1}
\int_0^\infty\int_0^\infty |e^{-ax}k(x,y)|^2dxdy  <  \infty 
\end{equation} such that
$$K_tf(x)=1_{(0,t)}(x)\int_0^t k(t-x,y)f(y)dy,\quad f\in \Hi.$$  
Conversely if $K_t$ is given by a measurable function $k(x,y)$ as above, satisfying (\ref{k}) and (\ref{k1}), 
then the $C_0$-semigroup $\{T_t=S_t+K_t\}$ satisfies the above conditions {\rm (i)} and {\rm (ii)}. 

As described in \cite[Section 6]{I}, Tsirelson's type III $E_0$-semigroups arising from off white noises (see \cite{T1}, \cite{T2}) are 
isomorphic to generalized CCR flows associated with $\{S_t\}$ and $\{T_t\}$ as above with the spectral density functions 
$|M(iy)|^2$. 

Now we describe a particular subclass of $C_0$-semigroups which are perturbations of the shift. 
We denote $x\wedge y=\min\{x,y\}$. 
It is proved in \cite[Lemma 5.1]{I} that for any $\varphi \in L^1_{\mathrm{loc}}[0,\infty)\cap 
L^2((0,\infty),x\wedge 1dx)$, there exists a unique $\psi$ and a positive real number $a$ satisfying 
$e_a \psi \in L^1(0,\infty)$ and 
$$\Lt{\psi}=\frac{\Lt{\varphi}}{1-\Lt{\varphi}}.$$ The following theorem is also proved in \cite[Theorem 5.2]{I}. 

\begin{theorem} Let $\varphi \in L^1_{\mathrm{loc}}[0,\infty)\cap 
L^2((0,\infty),x\wedge 1dx)$ and $\psi$ be defined as above.  
Then $k(x,y)$ defined by $$k(x,y)=\varphi(x+y) + \int_0^x \varphi(x+y-s)\psi(s)ds,$$ 
satisfies the conditions (\ref{k}), (\ref{k1}), and consequently gives rise to a $C_0$-semigroup 
$\{T_t\}$ satisfying conditions {\rm (i)} and {\rm (ii)} of Definition \ref{perturb}.
\end{theorem}

\medskip

We will often be using the following theorem \cite[Theorem 5.3]{I}.

\begin{theorem} Let $\varphi \in L^1_{\mathrm{loc}}[0,\infty)\cap 
L^2((0,\infty),x\wedge 1dx)$ and $\{T_t\}$ be the $C_0$-semigroup constructed from $\varphi$ as in the above theorem, 
which is a perturbation of the shift. 
Let $A_M$, with $M \in \HD_2$, be the generator of $\{T_t\}$ and let $q \in L^2(0,\infty)\setminus D(A)$ be the 
function satisfying $(1+z)\Lt{q}(z)=M(z)$. Then $q$ is continuous at $0$ and $M(z)=q(0)(1-\Lt{\varphi}(z)).$
\end{theorem}

\medskip
 
When $\varphi$ is a real function 
in $L^1_{\mathrm{loc}}[0,\infty)\cap L^2((0,\infty),x\wedge 1dx)$ and 
let $M=1-\Lt \varphi$, then we get a $C_0$-semigroup $\{T_t\}$ as described above, which preserves the real functions 
in $\Hi$. 
We will show in Section \ref{examples} that the generalized CCR flows $\alpha^\varphi$ associated with such pairs 
$(\{S_t\},\{T_t\})$ is type III if and only if $\varphi \notin L^2(0,\infty)$. 

In the last section we associate von Neumann algebras to bounded open sets in $(0, \infty)$, for each $E_0$-semigroup 
in this family. 
These local von Neumann algebras form an invariant for the $E_0$-semigroup. 
For a subfamily of the above constructed $E_0$-semigroups, we provide two different open sets for given two different 
$E_0$-semigroups, whose associated von Neumann algebra is a type III factor for the first $E_0$-semigroup and type I 
factor for the second $E_0$-semigroup. 
This way we show that there exists an uncountable family of type III $E_0$-semigroups which are mutually non-cocycle conjugate. 
\bigskip

%%%%%%%%%%%%%%%%%%%%%%%%%%%%%%%%%%%%%%%%%%%%%%%%%%%%%%%%%%%%%%%%%%%%%%%%%%%%%%%%%%%%%%%%%%%%%%%%%%%%%%%%%%%%%%
%%%%%%%%%%%%%%%%%%%%%%%%%%%%%%%%%%%%%%%%%%%%%%%%%%%%%%%%%%%%%%%%%%%%%%%%%%%%%%%%%%%%%%%%%%%%%%%%%%%%%%%%%%%%%%
%%%%%%%%%%%%%%%%%%%%%%%%%%%%%%%%%%%%%%%%%%%%%%%%%%%%%%%%%%%%%%%%%%%%%%%%%%%%%%%%%%%%%%%%%%%%%%%%%%%%%%%%%%%%%%

\section{Sum systems and Generalized CCR flows}\label{genccr} 

In this section we study the $E_0$-semigroup associated with the product
system constructed out of a given sum system. 
Fix a sum system $(\{G_{a,b}\}, \{S_t\})$ and let $(\{H_t\}, \{U_{s,t}\})$ be the product system constructed out of it. 
Denote $G =G_{0,\infty}$, $A_t = A_{t, \infty}$. 
For $x \in G$, define $x_s \in G_{0,s}$, $x_{s,t} \in G_{s,t}$, and $x_{t, \infty}\in G_{t, \infty}$ by the unique decomposition,
$$x = x_s + x_{s,t} + x_{t,\infty}.$$ 
 
We may consider $S_t$ as a bounded linear invertible map from $G$ onto $G_{t, \infty}$. 
Hence $(S_t^*)^{-1}$ is a well-defined bounded operator from $G$ onto $G_{t, \infty}$. 
When there is no confusion, by misusing the notation, we consider $(S_t^*)^{-1}$ as an element of $\B(G)$ itself. 
Define $T_t \in \B(G)$, by 
$$T_t =(A_t^*)^{-1} A_t^{-1}(S_t^*)^{-1}~~ \forall~t \in [0,\infty).$$

\begin{lemma}\label{st-1}
For any $y \in G_{0,\infty}$, the family $\{S_t^{-1}y_{t, \infty}\}_{t>0}$ converges to $y$, as $t \rightarrow 0$.
\end{lemma}

\begin{proof}
We have $$\|S_t^{-1}y_{t, \infty}-y\| =\|S_t^{-1}(y_{t,\infty}-S_ty)\|\leq 
\|S_t^{-1}\|(\|y_{t, \infty}-y\| + \|S_ty -y\|).$$ 
We claim that there exist positive numbers $c$ and $\varepsilon$ such that 
for all $x\in G$ and $t\in (0,\varepsilon)$ we have $||S_tx||\geq c\|x\|$. 
Indeed, if it were not the case, we would have a decreasing sequence $\{t_n\}_{n=1}^\infty$ of positive numbers 
converging to zero and a sequence $\{x_n\}_{n=1}^\infty$ in $G$ satisfying $\|x_n\|=1$ such that 
$\{\|S_{t_n}x_n \|\}_{n=1}^\infty$ converges to zero. 
Note that since $\{S_t\}$ is a $C_0$-semigroup, there exist positive constants $a$ and $b$ such that 
$\|S_t\|\leq ae^{bt}$ holds for all $t>0$ (see \cite[page 232]{Y}). 
We choose $s>t_1$. 
Then 
$$\|S_sx_n\|=\|S_{s-t_n}S_{t_n}x_n\|\leq \|S_{s-t_n}\|\|S_{t_n}x_n\|\leq ae^{bs}\|S_{t_n}x_n\|\to 0,\; (n\to\infty),$$
which is a contradiction because $S_s$ is invertible as a map from $G$ onto $G_{s,\infty}$. 
Therefore the claim is proved, and so $\|S_t^{-1}\|$ is bounded when $t$ tends to 0. 
On the other hand $S_ty$ converges to $y$ by the strong continuity of the semigroup $\{S_t\}$ and 
$y_{t , \infty}$ converges to $y$ (see \cite[Proposition 24]{pdct}), which finishes the proof. 
\end{proof}

\medskip

\begin{lemma}\label{Ttsemigroup}
$\{T_t\}$ forms a $C_0$-semigroup on $G$.
\end{lemma}

\begin{proof}
Note 
that $T_t x = (A_t^*)^{-1}(0 \oplus (S_t^*)^{-1}x)$, and therefore 
\begin{eqnarray}\label{ip}
\langle T_t x, y\rangle =\langle (0 \oplus (S_t^*)^{-1}x), A_t^{-1}y\rangle  = \langle x, S_t^{-1}y_{t,\infty} \rangle, & \forall & x,y \in 
G.\end{eqnarray} 
Hence for any 
$x, y \in G$, 
\begin{eqnarray*}
\langle T_s T_t x, y\rangle & = & \langle T_t x, 
S_s^{-1}(y_{s,\infty})\rangle\\
& =& \langle (S_t^*)^{-1} x, (S_s^{-1}y_{s, \infty})_{t, \infty} 
\rangle\\
& = & \langle (S_{s+t}^*)^{-1}x, S_s(S_s^{-1}y_{s, 
\infty})_{t, \infty} \rangle
\end{eqnarray*} 
Suppose $y_{s,\infty}=S_s y^\prime$ for some $y^\prime \in G$, then 
$S_s(S_s^{-1}y_{s,
\infty})_{t, \infty} = S_s y^\prime_{t, \infty}=y_{s+t, \infty}$. 
Hence 
we get 
\begin{eqnarray*}
\langle T_s T_t x, y\rangle & = & \langle (S_{s+t}^*)^{-1}x, y_{s+t, 
\infty}\rangle\\
&=& \langle T_{s+t}x, y\rangle
\end{eqnarray*} So $\{T_t\}$ forms a semigroup. Now equation (\ref{ip}), 
thanks to Lemma \ref{st-1}, will imply that the semigroup 
$\{T_t\}$ is weak and hence strongly continuous (see \cite[page 233]{Y}). 
So $\{T_t\}$ forms a $C_0$-semigroup.
\end{proof}

We say that the pair $(\{S_t\},\{T_t\})$ is associated with the sum system $(G_{a,b},S_t)$. 
\medskip

The $E_0$-semigroup associated with the product system $(H_t, U_{s,t})$ can 
be described in terms of these two semigroups, $\{S_t\}$ and $\{T_t\}$ as follows. 
Let  $H=\Gamma(G^\C)$. 

\begin{prop}\label{e0} Let the notation be as above. 
Then there is a unique $E_0$-semigroup $\alpha_t$ on $\B(H)$ satisfying
$$\alpha_t(W(x))=W(S_tx),~~\alpha_t(W(iy))=W(iT_ty),~x, y \in G.$$
Moreover the product system associated with this $E_0$-semigroup is the one constructed out of the sum system.
\end{prop}

\begin{proof}
Fix $t\geq 0$. Set $H_{t, 
\infty}=\Gamma(G_{t,\infty}^\C)$ and $U_t=\Gamma(A_t):H_t 
\otimes H_{t,\infty} \rightarrow H$.

The map $T \mapsto \Gamma(S_t) T \Gamma(S_t)^*$ is an isomorphism between 
$\B(H)$ and $B(H_{t, \infty})$. So we get a $*$-endomorphism $\alpha_t^0$ of 
$\B(H)$ defined by 
\begin{eqnarray}\label{barst}
T & \mapsto & 
U_t (1_{H_{0,t}} \otimes \Gamma(S_t)T 
\Gamma(S_t)^*)U_t^*
\end{eqnarray}
We claim that the above endomorphism is $\alpha_t$.  
Indeed, for $x,y\in G$ we have 
\begin{eqnarray*}\alpha_t^0(W(x+iy))
 &=& \Gamma(A_t)W(0\oplus (S_tx+iS_t^{-1*}y))\Gamma(A_t)^*\\
 &=&W(S_tx+iA_t^{-1*}(0\oplus S_t^{-1*}y))\\
 &=&W(S_tx+iT_ty),
\end{eqnarray*}
where we identify $1_{H_t}\otimes W(z)$ with $W(0\oplus z)$ for $z\in G_{t,\infty}^\C$.
Hence $\alpha_t$ defined in the proposition extends to a $*$-endomorphism on $\B(H)$, 
and we get a generalized CCR flow.  Since any endomorphism on $\B(H)$ is determined uniquely by its restriction to the Weyl algebra, $\alpha_t$ is the unique extension.

To determine the product system of this $E_0$-semigroup, we want to 
determine the family of Hilbert spaces $$ E_t = 
\label{Ht}\{T \in \B(H): \alpha
_t(X)T=TX~~~\forall~~X 
\in \B(H)\}.$$
Given a $\xi_t \in H_t$, we define a bounded operator $T_{\xi_t}$ 
on 
$H$ by the 
following prescription, $$T_{\xi_t}\xi=U_t(\xi_t \otimes 
\Gamma(S_t)\xi),~~\forall ~ \xi \in H.$$ It is easy to verify that 
$T_{\xi_t}$ defines a bounded operator, and 
that $$\langle T_{\xi_t^\prime}^*T_{\xi_t}\xi, \eta\rangle =\langle \xi_t, \xi_t^\prime 
\rangle_{H_t}1_H , 
~~\forall ~\xi_t, \xi_t^\prime \in H_t.$$
We claim that $T_{\xi_t} \in E_t$. In fact by the definition of $\alpha_t$ 
(see equation  (\ref{barst})) 
\begin{eqnarray*}
\alpha_t(X)T_{\xi_t} \xi& = & U_t (1_{H_{0,t}} \otimes \Gamma(S_t) X
\Gamma(S_t)^*)U_t^* U_t(\xi_t \otimes
\Gamma(S_t)\xi)\\
& =& U_t (\xi_t \otimes \Gamma(S_t)X \xi)\\
& = & T_{\xi_t}X\xi
\end{eqnarray*}
So the association $ \xi_t\mapsto T_{\xi_t}$ provides an isometry between 
$H_t$ into $E_t$. We only need to prove that this map is surjective. 
Suppose $T \in E_t$ be such that $T^*T_{\xi_t}=0 ~~\forall ~ \xi_t \in 
H_t$, that is $T^*$ vanishes on all vectors of the form $U_t(\xi_t \otimes 
\Gamma(S_t)\xi), ~\xi_t \in H_t, \xi \in H$. But such vectors forms 
a total subset of $H$, and hence we conclude that $T=0$. It is also easy 
to verify that the product structure is also preserved.
\end{proof}

The above proposition together with Lemma \ref{quasi-equivalence} imply 

\begin{cor} \label{sum to C_0} Let the notation be as above. 
Then $\{T_t\}$ is a perturbation of $\{S_t\}$. 
\end{cor}
\medskip

Now we investigate the reverse question. 
Let $G$ be a real Hilbert space and $H = \Gamma(G^\C)$. 
We assume that a $C_0$-semigroup $\{T_t\}$ is a perturbation of another $C_0$-semigroup $\{S_t\}$ acting on $G$. 
Our task is to determine the product system for the generalized CCR flow associated with the pair 
$(\{S_t\}, \{T_t\})$. 

Define $$G_{0,t} = \Ker(T_t^*),~ 
G_{(0,\infty)}=\overline{\bigcup_{t > 0} G_{0,t}},~ G_{a,b}=S_a(G_{0,b-a}).$$  
We verify that $G_{a,b} \subseteq G_{c,d} ~~\mbox{whenever}~~(a,b) \subseteq (c,d).$ 
This is same as saying $$S_c S_{a-c}(G_{0, b-a}) \subseteq S_c(G_{0, d-c}).$$ 
So we only need to verify that $S_{a-c}x \in G_{0,d-c}$, for any $x \in G_{0,b-a}$. 

First we assume that $d \neq \infty$. We verify that $S_{a-c}x \perp \Ran(T_{d-c})$ as follows. For $z  \in G$,
\begin{eqnarray*}
\langle S_{a-c}x, T_{d-c}z\rangle & = & \langle x , S_{a-c}^*T_{a-c}T_{d-a} z\rangle \\
& = & \langle x, T_{d-a}z \rangle\\
& = & \langle x, T_{b-a}T_{d-b}z \rangle\\
&=& 0. 
\end{eqnarray*}

Now let $d = \infty$. If $b \neq \infty$ then the above verification 
itself implies that, for any finite $d^\prime \geq b$, we have  $$S_c 
S_{a-c}(G_{0, b-a}) \subseteq
S_c(G_{0,d^\prime-c}) \subseteq S_c(G_{0,\infty}).$$  So we have to 
consider only the case when $b =\infty$, and equivalently we only have 
to show that $S_t(G_{0,\infty})\subseteq G_{0,\infty}$ for any $t \geq 
0$.  But this 
follows immediately as we have 
already shown that for any fixed $s \geq 0$, $S_s(G_{0,t})\subset 
G_{0,\infty}$ for all $s, t 
\geq 
0$. 

\medskip
Note that since $S_t$ has a left inverse $T_t^*$, the range of $S_t$ is closed and the 
operator $S_t$ from $G$ onto $\Ran(S_t)$ has a bounded inverse. 

\begin{lemma} \label{direct sum}
Let the notation be as above and $0<t<s\leq \infty$. Then 
\begin{itemize}
\item [$(\rm{i})$] We have $T_t^*G_{0,s}\subset G_{0,s-t}$.

\item [$({\rm ii})$] The two operators $S_tT_t^*$ and $1-S_tT_t^*$ are idempotents such that 
$\Ran(S_tT_t^*)=\Ker(1-S_tT_t^*)=\Ran(S_t)$ and  
$\Ran(1-S_tT_t^*)=\Ker (S_tT_t^*)=G_{0,t}$. 
In particular, the Hilbert space $G$ is a topological direct sum of $G_{0,t}$ and 
$\Ran(S_t)$.  

\item [$(\rm{iii})$] The Hilbert space $G_{0,s}$ is a topological direct sum of $G_{0,t}$ and $S_tG_{0,s-t}$. 
\end{itemize}
\end{lemma}

\begin{proof} (i) It is easy to verify the statement for finite $s$. 
Assume $x\in G_{0,\infty}$. 
Then there exists a sequence $\{x_n\}$ converging to $x$ such that $x_n\in G_{0,n}$. 
Thus for $n$ larger than $t$, we get $T_t^*x_n\in G_{0,n-t}$. 
Since $\{T_t^*x_n\}$ converges to $T_t^*x$, the statement holds. 

(ii)  follows from direct computation. 
(iii) follows from (i) and (ii). 
\end{proof}

\medskip

Let $P:G \rightarrow G_{0,\infty}$ be the orthogonal projection. 
We define $S^0_t$ and $T^0_t$ by 
$$S^0_t= PS_tP, ~~~T^0_t= PT_tP.$$
Then $\{S_t^0\}$ and $\{T_t^0\}$ are $C_0$-semigroups and one is a perturbation of the other. 
Indeed, it follows from the fact that the inclusion relation $S_t(G_{0,\infty})\subseteq G_{0,\infty}$ implies  
$(1-P)S_tP=0$ and Lemma \ref{direct sum},(i) implies $PT_t(1-P)=0$ for all $t>0$. 
\medskip

\begin{prop}\label{StTt} Let $G$ be a real Hilbert space and let $\{S_t\}$ and $\{T_t\}$ be $C_0$-semigroups 
acting on $G$ such that $\{T_t\}$ is a perturbation of $\{S_t\}$.  
Let $\{G_{s,t}\}$, $\{S_t^0\}$, and $\{T_t^0\}$ be as above. Then 
\begin{itemize}
\item[${\rm (a)}$] The system $(\{G_{a,b}\}, \{S_t^0\})$ forms a sum system. \medskip
\item[${\rm (b)}$] The pair of $C_0$-semigroups 
$(\{S_t^0\},\{T_t^0\})$ is associated with $(\{G_{a,b}\}, \{S_t^0\})$. 
In consequence, the product system for the generalized CCR flow arising from 
$(\{S_t^0\},\{T_t^0\})$ is isomorphic to the one arising from $(\{G_{a,b}\},\{S_t^0\})$.\medskip 
\item[${\rm (c)}$] 
The product system for the generalized CCR flow arising from $(\{S_t\}, \{T_t\})$ is 
isomorphic to the product system arising from $(\{G_{a,b}\}, \{S_t^0\})$. 
In consequence, the generalized CCR flow arising from the pair $(\{S_t\},\{T_t\})$ is cocycle conjugate to that arising 
from $(\{S^0_t\}, \{T_t^0\})$.
\end{itemize}
\end{prop}

\begin{proof}
(a) We have already shown  the axiom  (iii) of Definition \ref{sumsystem} and 
$G_{a,b} \subseteq G_{c,d}$ for $(a,b) \subseteq (c,d).$ 
Since $\{T_t^0\}$ is a perturbation of $\{S_t^0\}$, the operator $S_t^0$ has a left inverse ${T_t^0}^*$.  
Therefore the restriction of $S_t^0$ to $G_{0,s}$ is an invertible operator from $G_{0,s}$ onto $S_tG_{0,s}$ with 
the bounded inverse. 
Since $S_t^{0*}(T_t^0-S_t^0)=P-S_t^{0*}S_t^0$ is a Hilbert Schmidt operator, the axiom (i) is satisfied.

To prove the axiom (ii), it is enough if we prove that the operator 
$$A_{t, \infty}:G_{0,t}\oplus G_{t, \infty}\ni x\oplus y\mapsto x+y\in G_{0,\infty},$$ 
is in the class $\cS(G_{0,t}\oplus G_{t, \infty}, G_{0,\infty})$, for all $t \in(0,\infty)$.
Let $G_t'=\Ran(S_t)$. 
We will prove a stronger statement that the operator, 
$$A^\prime_t:G_{0,t}\oplus G^\prime_t\ni x\oplus y \mapsto x+y\in G,$$ 
is in the class $\cS(G_{0,t}\oplus G_t^\prime, G)$, for all $t \in (0,\infty)$. 
Thanks to Lemma \ref{direct sum}, the operator $A'_{t}$ has a bounded inverse. 
Let $x_1,x_2\in G_{0,t}$ and $y_1,y_2\in G_{0,\infty}$. 
Then 
\begin{eqnarray*}\lefteqn{
\inpr{{A'_t}^* A'_t(x_1\oplus S_ty_1)}{x_2\oplus S_ty_2}
=\inpr{x_1+S_ty_1}{x_2+S_ty_2}}\\
&=&\inpr{x_1}{x_2}+\inpr{S_ty_1}{S_ty_2}+\inpr{x_1}{S_ty_2}+\inpr{S_ty_1}{x_2}\\
&=&\inpr{x_1\oplus S_ty_1}{x_2\oplus S_ty_2}+\inpr{T_t(S_t^*-T_t^*)x_1}{S_ty_2}+\inpr{S_ty_1}{T_t(S_t^*-T_t^*)x_2},
\end{eqnarray*}
where we use $T_t^*x_1=T_t^*x_2=0$. 
This shows that $A'_t\in \cS(G_{0,t}\oplus G_t^\prime, G)$ and in consequence 
$A_{s,t}\in\cS(G_{0,s}\oplus G_{s,s+t},G_{0,s+t})$. 
Therefore the axiom (ii) holds and $(\{G_{a,b}\},\{S_t\})$ is a sum system.

(b) It suffices to verify 
$$T^0_t=(A_{t, 
\infty}^*)^{-1}(A_{t,\infty}^{-1})(S_t^{0*})^{-1},$$
where $S^0_t$ is regarded as an element of $\B(G_{0,\infty},G_{t,\infty})$.  
For 
$ y 
\in G$, let $y = y_t + y_{t, 
\infty}$ be the unique decomposition such that $y_t \in G_{0,t}, 
y_{t, \infty} \in G_{t,\infty}$. By a calculation 
we have already done in Lemma \ref{Ttsemigroup}, for $x, y \in G$, we have,
$$\langle (A_{t, \infty}^*)^{-1}(A_{t,\infty}^{-1})(S_t^*)^{-1}x, 
y 
\rangle = \langle x, S_t^{-1} y_{t, \infty} \rangle.$$ 
On the other hand \begin{eqnarray*}
\langle T^0_tx, y\rangle & = & \langle x, {T_t}^*(y_t + S_tS_t^{-1}y_{t, \infty})\rangle\\
& = & \langle x, S_t^{-1}y_{t,\infty}\rangle.
\end{eqnarray*} 
Therefore $(\{S^0_t\},\{T^0_t\})$ is associated with $(\{G_{s,t}\},\{S^0_t\})$. 
\bigskip

(c) We use the notation of the proof of (a). 
Proceeding exactly in the same way as in the proof of (b), by replacing $G_{t, \infty}$ with $G^\prime_t$, 
we can verify that 
$$T_t=({A_t^\prime}^*)^{-1}{A_t^\prime}^{-1}(S_t^*)^{-1},$$ with 
$(S_t^*)^{-1}$ regarded as an element in $\cS(G,G'_t)$. 
By replacing $\Gamma(G_{t, \infty}^\C)$ with 
$\Gamma({G_t^\prime}^\C)$, and by exactly imitating 
the proof of the Proposition \ref{e0}, we conclude that $$R \mapsto 
\Gamma(A_t^\prime)(1_{\Gamma(G_{0,t}^\C)} \otimes 
\Gamma(S_t)R\Gamma(S_t)^*)\Gamma(A_t^\prime)^*,$$ is the generalized CCR flow given by the pair $(\{S_t\}, \{T_t\})$.

If we again imitate the proof of Proposition \ref{e0}, the part where 
the product system is computed, we will be able to see that the 
product system associated with the generalized CCR flow given by the pair $(\{S_t\}, \{T_t\})$ also coincides with the product 
system constructed out of the sum system $(\{G_{a,b}\},\{S_t\})$. 

Since the product systems determine $E_0$-semigroups up to cocycle conjugacy, we conclude that
the generalized CCR flow given by $(\{S_t\},\{T_t\})$ is cocycle conjugate to the generalized CCR flow given by 
$(\{S^0_t\}, \{T_t^0\})$.
\end{proof}

\medskip
\begin{remark}\label{isometry}
It is practically impossible to classify pairs $(\{S_t\},\{T_t\})$ of $C_0$-semigroups acting on $G$ as above without posing 
any condition and let us consider the case where $\{S_t\}$ is a semigroup of isometries. 
Then up to unitary equivalence we may assume that $G=L^2((0,\infty),K)\oplus L$ and  
$S_t=S'_t\oplus U_t$, where $\{S'_t\}$ is the shift semigroup and $\{U_t\}$ is a 1-parameter unitary group. 
In this case, we have 
$$T_t=\left(
\begin{array}{cc}
T'_t &B_t  \\
 0&U_t 
\end{array}
\right),
$$ 
where $\{T'_t\}$ is a perturbation of $\{S'_t\}$. 
It is routine work to show that the two sum systems for $(\{S_t\},\{T_t\})$ and $(\{S'_t\},\{T'_t\})$ 
are isomorphic (though not identical in general), and so Proposition \ref{StTt} implies that the two generalized 
CCR flows arising from them are cocycle conjugate. 
Therefore it is worth investigating perturbations of the shift semigroup. 
This has been already done in \cite{I} in the case where $\dim K=1$ and we will analyze the structure of the resulting 
generalized CCR flows in Section 5-7. 
\end{remark}
\bigskip

%%%%%%%%%%%%%%%%%%%%%%%%%%%%%%%%%%%%%%%%%%%%%%%%%%%%%%%%%%%%%%%%%%%%%%%%%%%%%%%%%%%%%%%%%%%%%%%%%%%%%%%%%%%%%%%%%%%%%%%%%%%%%

\section{Type III criterion}

In this section we derive a necessary and sufficient condition, which would determine the type of 
the product system (in other words the type of the associated $E_0$-semigroup), arising from a divisible sum system with finite index. 
To begin with we define the notion of divisibility for sum systems and 
recall some results from \cite{pdct}. We first define addits 
for  a sum system, which were called as additive units in 
\cite{pdct}

\begin{definition} Let $(\{G_{a,b}\}, \{S_t\})$ be a sum system. 
A real addit for the sum system $(\{G_{(a,b)}\}, \{S_t\})$ is a
family $\{x_t\}_{t \in (0,\infty)}$ such that
$x_t \in G_{0,t} , ~ \forall ~t \in (0,  \infty)$, 
satisfying the following conditions.

\medskip \noindent
(i) The map $ t \mapsto \langle x_t, x\rangle$ is measurable for
any $ x \in G_{0,\infty}$.

\medskip \noindent
(ii)
$x_s
+ S_s x_t = x_{s+t}, ~~ \forall s, t,
\in (0, \infty),$
(i. e.) $A_{s,t}(x_s \oplus S_s x_t)=x_{s+t}$.

An imaginary addit for the sum system $(\{G_{a,b}\}, \{S_t\})$ is a
family $\{y_t\}_{t 
\in (0,\infty)}$ such that
$y_t \in G_{0,t} , ~ \forall ~t \in (0,  \infty)$, 
satisfying the following conditions.

\medskip \noindent
(i) The map $ t \mapsto \langle y_t, y\rangle$ is measurable for
any $ y \in G_{0,\infty}$.

\medskip \noindent
(ii) $\{y_t\}$ satisfies $(A_{s,t}^*)^{-1}(y_s \oplus (S_s^*)^{-1} y_t) =
y_{s+t},
~~
\forall s, t,
        \in (0, \infty).$
\end{definition}

\medskip

We denote by $R\cA\U$ and $I\cA\U$ the set of all real and 
imaginary addits respectively, which are real linear spaces. 
For a given real addit $\{x_t\}$, define $x_{s,t} 
= S_s(x_{t-s}) \in G_{s, t}$. Similarly for a given imaginary addit $\{y_t\}$ define $y_{s,t} 
= (S_s^*)^{-1}(y_{t-s}) \in G_{s, t}$.

We also define for an imaginary addit $\{y_t\}$, $$G_{0,s} \ni {}^s y^\prime_{s_1,s_2}= (A^*)^{-1}(0 \oplus y_{s_1, s_2} \oplus 0),
~\mbox{for any}~(s_1, s_2) \subset (0,s),$$ where $A:G_{0,s_1} \oplus G_{s_1,s_2} \oplus G_{s_2, s} \rightarrow G_{0,s}$ is defined by
$x\oplus y \oplus z \mapsto x + y+z$. 
It is easy to check that ${ }^sy^\prime_{s_1,s_2} \in \left(G_{0,s_1} \bigvee G_{s_2,
s}\right)^{\perp} \cap G_{0,s}.$ When $s=1$, we just denote ${}^ty'_{s_1,s_2}$by $y^\prime_{s_1, s_2}$, and 
$y_{0,t}^\prime$ by just $y_t^\prime$. Finally note that 
$$x_s+x_{s,s+t}= x_{s+t},~~  y_s^\prime +y_{s,s+t}^\prime = y_{s+t}^\prime.$$

\begin{definition} A sum system $(\{G_{a,b}\},\{S_t\})$ is called as a
divisible sum system if the addits exist and generate the sum system, (i. e.)
$$G_{0,s} =  \overline{\spa_\R  [x_{s_1,s_2}; (s_1,s_2) \subseteq 
(0,s), 
\{x_t\} \in
R\cA\U]}$$ and $$G_{0,s} =  \overline{\spa_\R  [{}^s 
y^\prime_{s_1,s_2}; 
(s_1,s_2) \subseteq
(0,s), \{y_t\} \in I\cA\U]}.$$
\end{definition}

\medskip

The following Proposition has been already proved in \cite{pdct}, except that $\{S_t\}$ need not be 
a semigroup of isometries. 
But in the proof of \cite[Proposition 37 (ii)]{pdct}, the verification of measurability of the function 
$t \rightarrow \langle x_t, y_t \rangle$ does not need this assumption, and also the relation 
$ \langle x_{s+t}, y_{s+t} \rangle = \langle x_s, y_s \rangle + \langle x_t, y_t \rangle,$ can also be 
verified without this assumption.

\begin{prop}\label{h}
Let $(\{G_{(a,b)}\},\{S_t\})$ be a divisible sum system. If $\{x_t\} \in R\cA\U$ and $\{y_t\}
\in I\cA\U$, then $$\langle x_t, y_t\rangle = \langle x_1, 
y_1\rangle t 
~~\forall~ t \in (0, \infty).$$ In general for any two intervals 
$(s_1,s_2),~(t_1, t_2) \subset (0, s)$, it is true that
\begin{eqnarray} 
\langle x_{s_1, s_2}, {}^s y^\prime_{t_1,t_2} \rangle =
\langle x_1, y_1 \rangle |
(s_1, s_2) \cap (t_1, t_2)|,\end{eqnarray}
where $|.|$ is the Lebesgue
measure on
$\R$.
\end{prop}

\medskip
The next lemma immediately follows from the definition of a divisible sum system and Proposition \ref{h}, 
which allows us to introduce the notion of the index of a divisible sum system. 

\begin{lemma}\label{non-degenerate} Let $(\{G_{a,b}\},\{S_t\})$ be a divisible sum system. 
For $x=\{x_t\}\in \RAU$ and $y=\{y_t\}\in IAU$, we set $b_G(x,y)=\inpr{x_1}{y_1}$. 
Then $b_G$ is non-degenerate as a bilinear form $b_G:\RAU\times \IAU\rightarrow \R$.   
\end{lemma}

\begin{definition}\label{index} For a divisible sum system $(\{G_{a,b}\},\{S_t\})$, 
the index $\mathrm{ind}\: G$ is the number $\dim \RAU=\dim \IAU\in \N\cup\{\infty\}$. 
\end{definition}

From now onwards we assume that $(\{G_{a,b}\}, \{S_t\})$ is a divisible sum system. 
We further assume that $\mathrm{ind}\: G=n$ is finite. 
In that case, both $\RAU$ and $\IAU$ carry unique linear topologies.  
Denote 
$$G^0_{0,t}=\spa_\R  [x_{s_1,s_2}; (s_1,s_2) \subseteq
(0,t),\ \{x_t\}\in \RAU]\subseteq G_{0,t},$$ 
$${G^0_{0,t}}^\prime=\spa_\R  [{}^ty^\prime_{s_1,s_2} ;(s_1,s_2) \subseteq(0,t),\;\{y_t\}\in \IAU]\subseteq G_{0,t}.$$ 
For a given linear map $J:\RAU \rightarrow \IAU$, we set $J_{t,0}$ to be the linear map 
$J_{t,0}:G^0_{0,t}\rightarrow {G^0_{0,t}}^\prime$ determined by 
$$J_{t,0}(x_{s_1, s_2})={}^tJ(x)^\prime_{s_1, s_2},$$ 
for $(s_1, s_2) \subseteq (0,t)$ and $x\in \RAU$.  
When $J_{t,0}$ has a bounded extension to $G_{0,t}$ we denote it by $J_t$. 

We need the following lemma.

\begin{lemma}\label{bdd}
Let $G$ be a real Hilbert space and let $R_0$ be a real linear operator on $G^\C$ with dense domain $D(R_0)$, 
which preserves imaginary parts of the inner product. 
Suppose there exists a unitary operator $U \in \B(\Gamma(G^\C))$ satisfying 
\begin{eqnarray}\label{intertwine} UW(x)U^*& = & W(R_0x),~~\forall~x \in D(R_0),\end{eqnarray} 
then $R_0$ extends to a bounded invertible operator $R$ on $G^\C$. 
Further it is true that $R \in \cS(G \oplus G, G \oplus G)$, where we identify $G\oplus G$ with $G^\C$ equipped with 
the real inner product $\inpr{\cdot}{\cdot}_\R=\RE\inpr{\cdot}{\cdot}$.  
\end{lemma}

\begin{proof} 
First let us prove that $R_0$ is bounded. 
Suppose $\{x_n\}\subseteq D(R_0)$ be any sequence which converges to $0$. 
Then, by the strong continuity of the Weyl representation, we conclude that $W(x_n)$ converges strongly to $1$. 
Therefore, by our assumption (\ref{intertwine}), it follows that 
$W(R_0x_n)$ also converges strongly to $1$. 
Consequently $\langle W(R_0x_{n})\Phi,\Phi\rangle = e^{-\|R_0x_n\|^2/2}$ converges to $1$ and hence we conclude that $R_0x_n$ 
also converges to $0$. 

We claim that the range of $R_0$ is dense in $G^\C$. 
Let $K=\Ran(R_0)^\perp$ with respect to the real inner product. 
Then \cite[Theorem 1,(5)]{Ara1} shows 
$$\{W(R_0z);\; z\in G^\C\}'=\{W(z);\; z\in iK\}''.$$
Since the left-hand side is $\C$ thanks to $W(R_0z)=UW(z)U^*$, 
we get $K=\{0\}$, which shows the claim. 
Note that $R_0$ is automatically injective and the inverse of $R_0$ is well defined as 
a densely defined operator. 
A similar argument as above shows that the inverse is also bounded. 
The remaining  part follows from the 
converse statement in the original Shales theorem (\cite{KRP}). 
\end{proof}

\medskip

We will also be using the following  lemmas.

\medskip

\begin{lemma}\label{almost coboundary} Let $G$ be a second countable locally compact abelian group 
and let $(w,K)$ be a continuous unitary representation of $G$ on a complex Hilbert space $K$ without containing 
the trivial representation. 
Let $c:G\rightarrow K$ be a continuous 1-cocycle, that is, the map $c$ satisfies the cocycle relation 
$$c(r+s)=c(r)+w(r)c(s),\quad \forall ~ r,s\in G.$$
Then the following equation holds for all $r,s\in G$: 
$$\inpr{c(r)}{w(r)c(s)}=\inpr{c(s)}{w(s)c(r)}.$$
\end{lemma}

\begin{proof} Thanks to \cite[Theorem 4.2.1]{S}, there exists a sequence $\{z_n\}$ in $K$ such that 
$\{z_n-w(g)z_n\}$ uniformly converges to $c(g)$  on every compact subset of $G$. 
Then we have 
\begin{eqnarray*}\lefteqn{
\inpr{c(r)}{w(r)c(s)}=\lim_{n\to\infty}\inpr{z_n-w(r)z_n}{w(r)z_n-w(r)w(s)z_n}}\\
&=&\lim_{n\to\infty}\big(\inpr{z_n}{w(r)z_n}+\inpr{z_n}{w(s)z_n}-\inpr{z_n}{z_n}-\inpr{z_n}{w(r)w(s)z_n}\big)
\end{eqnarray*}
On the other hand, 
\begin{eqnarray*}
\lefteqn{
\inpr{c(s)}{w(s)c(r)}=\lim_{n\to\infty}\inpr{z_n-w(s)z_n}{w(s)z_n-w(s)w(r)z_n}}\\
&=&\lim_{n\to\infty}\big(\inpr{z_n}{w(s)z_n}+\inpr{z_n}{w(r)z_n}-\inpr{z_n}{z_n}-\inpr{z_n}{w(s)w(r)z_n}\big),
\end{eqnarray*}
which shows the statement. 
\end{proof}

\medskip

Let $K$ be a complex Hilbert space. 
Recall that the automorphism group $G_K$ of the exponential product system of index $\dim K$ is described as 
follows (see \cite{Arv}): 
Let $U(K)$ be the unitary group of $K$. 
Then $G_K$ is homeomorphic to $\R\times K\times U(K)$ with the group operation 
$$(a,\xi,u)\cdot(b,\eta,v)=(a+b+\IM\inpr{\xi}{u\eta},\xi+u\eta,uv). $$
For $(a,\xi,u)\in G_K$, the corresponding automorphism is realized by the family of unitary operators 
$e^{i at}W(1_{(0,t]}\xi)\EXP( 1_{(0,t]}u)$, $t>0$. 
Direct computation shows the following: 

\begin{lemma}\label{gauge group} Let $G$ be an abelian group and let $\rho:G\ni r\mapsto (a(r),\xi(r),u(r))\in G_K$ be a map. 
Then $\rho$ is a homomorphism if and only if the following relation holds for every $r,s\in G$: 
\begin{equation}\label{g1}a(r+s)=a(r)+a(s)+\IM\inpr{\xi(r)}{u(r)\xi(s)},\end{equation}
\begin{equation}\label{g2} \xi(r+s)=\xi(r)+u(r)\xi(s),\end{equation}
\begin{equation}\label{g3} u(r+s)=u(r)u(s).\end{equation}
In particular, when $u(r)=1$ for all $r\in G$, then $\rho$ is a homomorphism if and only if 
\begin{equation}\label{g4} a(r+s)=a(r)+a(s),
\end{equation}
\begin{equation}\label{g5} \xi(r+s)=\xi(r)+\xi(s),
\end{equation}
\begin{equation}\label{g6} \IM \inpr{\xi(r)}{\xi(s)}=0.
\end{equation}
\end{lemma}

\begin{proof} The first statement is obvious. 
Assume $u(r)=1$ now. 
Then Equation (\ref{g1}) implies 
$$a(r+s)-a(r)-a(s)=\IM\inpr{\xi(r)}{\xi(s)}.$$
Note that the left-hand side is symmetric in $r$ and $s$ while the right-hand side is anti-symmetric. 
Thus the second statement holds. 
\end{proof}
\medskip

\begin{theorem}\label{unitless} Let $(\{G_{a,b}\}, \{S_t\})$ be a divisible
sum system of finite index and let $(\{H_t\}, \{U_{s,t}\})$ be the product 
system constructed out of the above sum system. 
Then the following statements are equivalent.
\begin{itemize}
\item[{\rm (i)}] The product system $(H_t, U_{s,t})$ is of type $I$.
\item[{\rm (ii)}] There exists a linear isomorphism $J:\RAU\rightarrow \IAU$ satisfying the following property: 
the bilinear form $b_G(\cdot,J\cdot)$ is an inner product of $\RAU$, and for each $t>0$, the operator $J_{t,0}$ extends to 
a bounded operator $J_t$ on $G_{0,t}$ such that $J_t \in \cS(G_{0,t}, G_{0,t})$. 
\item[{\rm (iii)}] There exists a linear isomorphism $J:\RAU\rightarrow \IAU$ satisfying the following property: 
the bilinear form $b_G(\cdot,J\cdot)$ is an inner product of $\RAU$ and the operator $J_{1,0}$ extends to a bounded operator 
$J_1$ on $G_{0,1}$ such that $J_1 \in \cS(G_{0,1}, G_{0,1})$. 
\end{itemize}
\end{theorem}

\begin{proof} We set $n=\mathrm{ind}\; G$. 

(i) $\Rightarrow$ (ii)
We assume that the product system $(H_t, U_{s,t})$ is of type $I$.
Then it is isomorphic to an exponential product system (\cite{Arv}). 
We first claim that this exponential product system is also of index $n$. 
For $t >0$, let $$V_t:\Gamma(G_{0,t})\rightarrow \Gamma(L^2((0,t), K))$$ 
be a family of unitary operators, implementing the isomorphism between the 
above product systems.  Here $K$ is some separable complex Hilbert space. 
We want to show that the dimension of $K$ is $n$. 

For each $x \in \RAU$ and $y\in \IAU$, the families $\{W(x_t)\}_{t>0}$ and $\{W(iy_t)\}_{t>0}$ form automorphisms for 
the product system  $(H_t, U_{s,t})$ (\cite[Theorem 26]{pdct}) satisfying the relations:
\begin{equation}\label{u1}
W(x^{(1)}_t)W(x^{(2)}_t)=W(x^{(1)}_t+x^{(2)}_t),\quad \forall ~ x^{(1)},x^{(2)}\in \RAU,
\end{equation}
\begin{equation}\label{u2}
W(iy^{(1)}_t)W(iy^{(2)}_t)=W(iy^{(1)}_t+iy^{(2)}_t),\quad \forall~ y^{(1)},y^{(2)}\in \IAU 
\end{equation}
\begin{equation}\label{u3}
W(x_t)W(iy_t)=e^{-2itb_G(x,y)}W(iy_t)W(x_t),\quad \forall~ x\in \RAU,\;\forall ~ y\in \IAU. 
\end{equation}
Therefore there exists two continuous homomorphisms 
$$\rho:\RAU\ni x\mapsto (a(x),\xi(x),u(x))\in G_K,$$
$$\sigma:\IAU\ni y\mapsto (b(y),\eta(y),v(y))\in G_K,$$
satisfying 
\begin{equation}\label{u4}
V_tW(x_t)V_t^*=e^{ita(x)}W(1_{(0,t]}\xi(x))\EXP(1_{(0,t]}u(x)),
\end{equation}
\begin{equation}\label{u5}
V_tW(iy_t)V_t^*=e^{itb(y)}W(1_{(0,t]}\eta(y))\EXP(1_{(0,t]}v(y)),
\end{equation}
where $1_{(0,t]}$ denotes the characteristic function of the interval $(0,t]$. 
Equation (\ref{u3}) implies that in addition to the relations in Lemma \ref{gauge group}, we have 
\begin{equation}\label{u6}
2b_G(x,y)=\IM\inpr{\eta(y)}{v(y)\xi(x)}-\IM\inpr{\xi(x)}{u(x)\eta(y)},\end{equation}
\begin{equation}\label{u7}\xi(x)+u(x)\eta(y)=\eta(y)+v(y)\xi(x),\end{equation}
\begin{equation}\label{u8} u(x)v(y)=v(y)u(x).\end{equation}

Let $w((x,y))=u(x)v(y)$ and $c(x,y)=\xi(x)+u(x)\eta(y)$. 
Then (\ref{g3}) and (\ref{u8}) imply that $(K,w)$ is a continuous unitary representation of $\RAU\times \IAU$, and (\ref{g2}) and 
(\ref{u7}) imply that $c$ is a continuous $1$-cocycle. 
Let $$K_0=\{z\in K;\; w(g)z=z,\;\forall g\in \RAU\times \IAU\},$$ 
and let $K_1=K_0^\perp$. 
Let $\xi_i(x)$ be the projection of $\xi(x)$ to $K_i$ and let $\eta_i(y)$ be the projection of $\eta(y)$ to $K_i$. 
Then Lemma \ref{almost coboundary} implies 
\begin{equation}\label{u9}
\inpr{\xi_1(x)}{u(x)\eta_1(y)}=\inpr{\eta_1(y)}{v(y)\xi_1(x)}, 
\end{equation}
and Equation (\ref{u6}) is equivalent to 
\begin{equation}\label{u10}
b_G(x,y)=\IM\inpr{\eta_0(y)}{\xi_0(x)}. 
\end{equation}

Assume that $K_1$ is not trivial. 
Let $0<p<q<t$. 
Then it is routine work to show  
$$V_tW(x_{p,q})V_t^*=e^{i(q-p)a(x)}W(1_{(p,q]}\xi(x))\EXP(1_{(0,p]}+1_{(p,q]}u(x)+1_{(q,t]}),$$
$$V_tW(i{}^ty_{p,q}')V_t^*=e^{i(q-p)b(y)}W(1_{(p,q]}\eta(y))\EXP(1_{(0,p]}+1_{(p,q]}v(y)+1_{(q,t]}).$$
By definition of $K_1$, either $u(x^0)$ or $v(y^0)$ is not trivial for some $x^0\in \RAU$ and $y^0\in \IAU$. 
Thus we assume that $u(x^0)\neq 1$ (the case with non-trivial $v(y^0)$ can be treated in the same way). 
Direct computation using Lemma \ref{almost coboundary} and (\ref{g1}) shows that the operator 
$W(1_{(0,t]}\xi_1(x^0))\EXP(1_{(0,t]}u(x^0))$ commutes with $V_tW(x_{p,q})V_t^*$ and $V_tW(i{}^ty'_{p,q})V_t^*$ 
for all $x\in \RAU$, $y\in \IAU$, and $0<p<q<t$. 
However, this contradicts the irreducibility of the vacuum representation of the Weyl algebra, since the sets 
$\{x_{p,q}; (p,q) \subseteq (0,t),\;x\in \RAU\}$ and $\{{}^ty^\prime_{p,q}; (p,q) \subseteq (0,t),\; y\in \IAU\}$ 
are total in $G_{0,t}$ due to the divisibility of the sum system. 
Hence $K=K_0$. 

Now assume that $\dim K$ is strictly larger than $n$. 
Then there exists non-zero $\zeta\in K$ orthogonal to $\xi(\RAU)$ and $\eta(\IAU)$ with respect to 
the real inner product $\RE\inpr{\cdot}{\cdot}$. 
Again we can show that $W(i1_{(0,t]}\zeta)$ would commutes with $V_tW(x_{p,q})V_t^*$ and $V_tW(i{}^ty'_{p,q})V_t^*$, for all $x\in \RAU$, $y\in \IAU$, and $0<p<q<t$, 
which is a contradiction. 
Therefore  we conclude that $\dim K\leq n$. 

In the above argument, we have shown the following: there exist continuous homomorphisms 
$\xi:\RAU\rightarrow K$, $\eta:\IAU\rightarrow K$, $a:\RAU\rightarrow \R$, and $b:\IAU\rightarrow \R$ satisfying 
\begin{equation}\label{v1}
V_tW(x_t)V_t^*=e^{ita(x)}W(1_{(0,t]}\xi(x)),\end{equation}
\begin{equation}\label{v2} V_tW(iy_t)V_t^*=e^{itb(y)}W(1_{(0,t]}\eta(y)),\end{equation}
\begin{equation}\label{v3}
\IM\inpr{\xi(x^{(1)})}{\xi(x^{(2)})}=0,\quad \forall ~x^{(1)},x^{(2)}\in \RAU,
\end{equation}
\begin{equation}\label{v4}
\IM\inpr{\eta(y^{(1)})}{\eta(y^{(2)})}=0,\quad \forall ~y^{(1)},y^{(2)}\in \IAU,
\end{equation}
\begin{equation}\label{v5}
b_G(x,y)=\IM \inpr{\eta(y)}{\xi(x)},\quad \forall~ x\in \RAU,\;\forall ~y\in\IAU.  
\end{equation}
Since $b_G$ is non-degenerate, Equation (\ref{v5}) shows that $\xi$ and $\eta$ are injective, and in particular, 
the real dimension of the image of $\xi$ is $n$. 
Thus Equation (\ref{v3}) shows that there exists an orthonormal basis $\{e_j\}_{j=1}^n$ of $K$ consisting of 
elements in $\xi(\RAU)$, and so $\dim K=n$. 
Let $x^j=\xi^{-1}(e_j)$. 
Then $\{x^j\}_{j=1}^n$ is a basis of $\RAU$.  
Let $\{y^j\}_{j=1}^n$ be the dual basis of $\{x^j\}_{j=1}^n$ in $\IAU$  with respect to the bilinear form $b_G$. 
Then there exists a real matrix $(\lambda_{jk})$ such that 
$$\eta(y^j)=ie_j+\sum_{k=1}^n\lambda_{jk}e_k.$$
Moreover, Equation (\ref{v4}) implies that the matrix $(\lambda_{jk})$ is symmetric. 
Thus by changing the basis $\{e_j\}_{j=1}^n$ if necessary, we may and do assume that $(\lambda_{jk})$ is diagonal and 
there exist real numbers $\lambda_j$ such that 
\begin{equation}\label{v6}
\eta(y^j)=(\lambda_j+i)e_j.
\end{equation}

Next we want to get rid of $a$ and $b$ in the above. 
Let 
$$\zeta=\frac{1}{2}\sum_{j=1}b(y^j)\xi(x^j)-\frac{1}{2}\sum_{j=1}^na(x^j)\eta(y^j).$$
Then direct computation yields $\IM \inpr{\zeta}{\xi(x^k)}=-a(x^k)/2$ and 
$\IM \inpr{\zeta}{\eta(y^k)}=-b(y^k)/2$. 
Now by replacing $V_t$ with $W(1_{(0,t]}\zeta)V_t$, we may and do assume $a$ and $b$ are trivial. 

Let $\lambda_j+i=r_je^{i\theta_j}$ with $r_j>0$ and $\theta_j\in (0,\pi)$. 
Let $g\in U(K)$ be the unitary operator determined  by $ge_j=e^{i\theta_j}e_j$ for all $j$. 
We denote $U_t = V_t^* \EXP(1_{(0,t]}g)V_t$. 
Then for $0<p<q<t$, we have
$$U_t W(x^j_{p,q})U_t^*=V_t^*W(\frac{1}{r_j}1_{(p,q]}\eta(y^j))V_t=W(i\frac{1}{r_j}{}^t{y^{j}}'_{p,q}),$$
\begin{eqnarray*}
U_tW(i{}^t{y^j}'_{p,q})U_t^*&=&V_t^*W(1_{(p,q]}r_je^{i2\theta_j}e_j)V_t\\
 &=& V_t^*W(1_{(p,q]}(2\cos\theta_jr_je^{i\theta_j}-r_j)e_j)V_t\\
 &=&V_t^*W(1_{(p,q]}(\xi(-rx^j)+\eta(2\cos\theta_jy^j)))V_t\\
 &=&e^{-i2(q-p)r_j\cos\theta_j} V_t^*W(1_{(p,q]}\xi(-r_jx^j))V_tV_t^*W(1_{(p,q]}\eta(2\cos\theta_jy^j))V_t\\
 &=&e^{-i2(q-p)r_j\cos\theta_j }W(-r_jx^j_{p,q})W(i2\cos\theta_j {}^t{y^j}'_{p,q})\\
 &=&W(-r_jx_{p,q}+i2\cos\theta_j {}^t{y^j}'_{p,q}).
\end{eqnarray*}
We introduce a real linear operator $R_0$ with domain $D(R_0)=G_{0,t}^0+i{G^{0}_{0,t}}'$ by setting
$$R_0x^j_{p,q}=\frac{i}{r_j}{}^t{y^j}'_{p,q},\quad R_0i{}^t{y^j}'_{p,q}= -r_jx_{p,q}+i2\cos\theta_j {}^t{y^j}'_{p,q},$$
which preserves the imaginary part of the inner product. 
Then we have the relation $U_tW(x)U_t^*=W(R_0x)$ for every $x\in D(R_0)$. 
Let $J:\RAU\rightarrow \IAU$ be the linear operator determined by $Jx^j=r_j^{-1}y^j$. 
Lemma \ref{bdd} implies that $R_0$  extends to an operator $R \in \cS(G_{0,t} \oplus G_{0,t}, G_{0,t} \oplus G_{0,t}),$ 
and in consequence $J_t$ exists.  
Moreover, the operator $J_t$ is invertible and $R$ can be expressed in a matrix form as, 
$$R=\left(
\begin{array}{cc}
 0& -J_t^{-1} \\
 J_t&B 
\end{array}
\right),
$$
where $B$ is the bounded linear extension of the map $i{}^t{y^j}'_{p,q}\mapsto 2\cos\theta_ji{}^t{y^j}'_{p,q}$. 
Thus  
$$R^*R=\left(
\begin{array}{cc}
J_t^*J_t &J_t^*B  \\
 BJ_t& J_t^{-1*}J_t^{-1}+B^2 
\end{array}
\right).
$$
Therefore $J_t\in \cS(G_{0,t}, G_{0,t})$ and $B$ is a Hilbert-Schmidt operator, and  in consequence, we have 
$\cos\theta_j=0$ for all $j$. 
This shows $Jx^j=y^j$ and $b_G(x^j,Jx^k)=\delta_{j,k}$, and so $b_G(\cdot,J\cdot)$ is an inner product 
(with an orthonormal basis $\{x^j\}_{j=1}^n$).

\bigskip

\noindent(ii) $\Rightarrow$ (iii) Clear.

\bigskip

\noindent
(iii) $\Rightarrow$ (i) Assume that there exists a linear isomorphism $J:\RAU\rightarrow \IAU$ such that 
$b_G(\cdot,J\cdot)$ gives an inner product of $\RAU$ and that the bounded extension $J_1$ of $J_{1,0}$ exists and 
$J_1\in \cS(G_{0,1})$. 
We first claim $J_t \in \cS(G_{0,t})$ for each $t\in (0,1)$. 
Denote $G_{0,t}^\prime = G_{t,1}^{\perp} \cap G_{0,1}$. Notice that 
$J_1|_{G_{0,t}}$ maps $G_{0,t}$ 
to   
$G_{0,t}^\prime$ and $A_{t,1-t}^*$ maps $G_{0,t}^\prime$ to $G_{0,t}$. Now 
our claim follows immediately from the observation that 
$$J_{t,0}=(A_{t,1-t}^*|_{G_{0,t}^\prime})(J_{1,0}|_{G_{0,t}}).$$ 
With the claim, it is routine work to show that $J_t$ exists and $J_t\in \cS(G_{0,1})$ for all $t>0$. 

Since $b_G(\cdot,J\cdot)$ gives an inner product of $\RAU$, we choose an orthonormal basis $\{x^j\}_{j=1}^n$ with respect to 
this inner product. 
We set $y^j=Jx^j$. 
Let $\cK$ be an $n$-dimensional real Hilbert space with an orthonormal basis $\{e_j\}_{j=1}^n$. 

Denote $$L^{(0)}((0,t),\cK)=\spa_\R  \{ 1_{(p, q]}\xi:(p, q) \subseteq (0,t),\;\xi\in \cK \}$$ 
which is a dense subspace of the real Hilbert space $L^2((0,t),\cK)$. 
Define 
$$B_t:G^0_{0,t} \rightarrow L^2((0,t),\cK), ~~~ B^\prime_t:L^{(0)}((0,t),\cK) \rightarrow G_{0,t}^\prime$$ 
by 
$$B_t(x^j_{p, q})=1_{(p,q]}e_j, ~~~B_t^\prime(1_{(p,q]}e_j)={}^t{y^j}^\prime_{p,q},$$ and extend by linearity. 
Notice that $J_{t,0}=B_t^\prime B_t$. Using Proposition \ref{h}, we get 
$$\langle B_tx^j_{p,q}, 
1_{(r,s]}e_k \rangle =\delta_{j,k}|[p,q]\cap[r,s]|=\langle x^j_{p,q}, B_t^\prime 1_{(r, s]}e_k \rangle.$$ 
By taking linear sums, we observe that $B_t$ and $B_t^\prime$ satisfy the adjoint condition on a dense subspace. 
So for any $x \in G^0_{0,t}$ we have 
$$\|B_tx\|^2 =\langle B_t^\prime B_tx, x \rangle \leq \|J_t\|\|x\|^2.$$ 
which shows that $B_t$ extends to an element in $\B(G_{0,t},L^2((0,t),\cK))$. 
The operator $B_t'$ also extends to a bounded operator because it is restriction of $B_t^*$. 
We use the same symbols $B_t$ and $B_t'$ for their bounded extension. 
Then we have $B_t^*B_t=J_t.$ 

Now from our assumption $J_t \in \cS(G_{0,t}, G_{0,t})$ we have $B_t \in \cS(G_{0,t}, L^2((0,t),\cK)).$ 
It is routine verification to see that $B_t$ satisfies $$B_s 
\oplus S_s^\prime B_t = B_{s+t}A_{s,t}(1_{G_{0,s}} \oplus S_s|_{G_{0,t}}) ~~\forall ~~s,t \in (0,1),$$ where 
we have identified $L^2((0,s),\cK)\oplus L^2((s, s+t),\cK)$ with $L^2((0,s+t),\cK)$. 
Therefore $(\{G_{a,b}\}, \{S_t\})$ is isomorphic, as sum system, to $(\{L^2((a,b),\cK)\}),\{S'_t\})$ where 
$\{S'_t\}$ is the unilateral shift. 
So $(\{H_t\}, \{U_{s,t}\})$ is isomorphic to the exponential product system, and hence type I. 
\end{proof}

\medskip

\begin{remark}
It has been proved in \cite[Theorem 39]{pdct} that only type I and type III product systems can be constructed from divisible sum systems. 
So thanks to the above Theorem, violating the condition $J_1 \in \cS(G_{0,1}, G_{0,1})$ is necessary and sufficient for 
the associated product system to be of type III. 
This criterion is much more powerful than the necessary condition for type I already proved in \cite[Theorem 40]{pdct}. 
In fact we can arrive at that condition just by assuming that $J_1$ is bounded. 
Suppose that $E_n \subseteq (0,1)$ be a sequence of elementary sets satisfying $ \liminf{G_{E_n}} =G_{0,1}$. 
That is, given any $x \in G_{0,1}$ there exists a sequence $\{x_n\}$ satisfying $x_n \in G_{E_n}$ and $x_n \rightarrow x$. 
If we set $y_n =J_1x_n$, then $y_n \in G_{E_n^c}^{\perp}$. 
Also if we assume that $J_1$ is bounded, then $y_n$ converges to $J_1x$. 
Thanks to \cite[Lemma 28]{pdct} this would imply that $\limsup{G_{E_n^c}} =\{0\}.$    
In Section 6, we will see that there are examples of divisible sum systems of finite index with bounded $J_1$, which give rise to type III 
product systems. 
\end{remark}

\bigskip

%%%%%%%%%%%%%%%%%%%%%%%%%%%%%%%%%%%%%%%%%%%%%%%%%%%%%%%%%%%%%%%%%%%%%%%%%%%%%%%%%%%%%%%%%%%%%%%%%%%%%%%%%%%%%%%%%
%%%%%%%%%%%%%%%%%%%%%%%%%%%%%%%%%%%%%%%%%%%%%%%%%%%%%%%%%%%%%%%%%%%%%%%%%%%%%%%%%%%%%%%%%%%%%%%%%%%%%%%%%%%%%%%%%
%%%%%%%%%%%%%%%%%%%%%%%%%%%%%%%%%%%%%%%%%%%%%%%%%%%%%%%%%%%%%%%%%%%%%%%%%%%%%%%%%%%%%%%%%%%%%%%%%%%%%%%%%%%%%%%%%
\section{Additive cocycles}\label{cocycles}

In this section we compute the addits when $G=\Hi_\R$ and $\{S_t\}$ is the shift semigroup.  
We use the notation $\LR t$ for the set of all real functions in $L^2(0,t)$.  
We follow the notations in \cite{I} and use many results from there. 
The reader may refer to the summary of notations and results reviewed at the end of Section \ref{pre}. 

We denote by $\HD_{2,\R}$ the set of functions $M\in \HD_2$ 
satisfying $\overline{M(z)}=M(\overline{z})$. 
We fix $M(z)$ in $\HD_{2,\R}$ and set  
$T_t=e^{tA_M}$, and $K_t=T_t-S_t$.   
With the above condition, the real part $\HiR$ is preserved by $T_t$.  

As defined in Section \ref{genccr}, for $t>0$, we set 
$$G^M_{0,t}=\{f\in \Hi_\R;\; T_t^*f=0\}$$ 
and set $G^M_{0,\infty}$ to be  $\overline{\bigcup_{t>0}G^M_{0,t}}$. 
When $M$ is an outer function, one can see that $G^M_{0,\infty}=\HiR$  
from \cite[Theorem 6.4]{I}. 

\begin{definition}\label{additive cocycle}
A measurable family $\{c_t\}_{t>0}$ of elements in $\Hi_\R$ is 
said to be a real additive cocycle for the pair $(\{S_t\}, \{T_t\})$, if $c_t\in G^M_{0,t}$ for all $t>0$, 
and the cocycle relation $c_{s+t}=c_s+S_sc_t$ holds for all $s,t>0$. 
A measurable family $\{d_t\}_{t>0}$ of elements in $\Hi_\R$ is 
said to be an imaginary additive cocycle for the pair $(\{S_t\}, \{T_t\})$, if $d_t\in \LR t$ for all 
$t>0$, and the cocycle relation $d_{s+t}=d_s+T_sd_t$ holds. 
\end{definition}

\begin{remark}\label{dtyt} Clearly the real additive cocycles are same as the real addits for the sum system $(\{G^M_{a,b}\},\{S_t\})$, 
as defined in Section \ref{genccr}. 

While proving part (a) in Proposition \ref{StTt}, it has been proved, for any $t\geq 0$, that the map 
$A^\prime_t:G^M_{0,t} \oplus L^2(t,\infty)_\R \mapsto L^2(0,\infty)_\R$, given by $x \oplus y \mapsto x +y$, is in 
$\cS(G^M_{0,t} \oplus L^2(t,\infty)_\R, L^2(0,\infty)_\R)$. 
If $\{y_t\} \in G^M_{0,t}$ is an imaginary addit for the sum system $(\{G^M_{a,b}\},\{S_t\})$, then we have 
$$\langle ({A_t^\prime}^*)^{-1}y_t, S_tx\rangle=\langle y_t \oplus 0, 0 \oplus S_t x\rangle= 0,~~\forall ~ x \in L^2(0,\infty),$$ 
and hence $S_t^*({A_t^\prime}^*)^{-1}y_t=0$. 
Moreover \begin{eqnarray*}({A_s^\prime}^*)^{-1}y_s + T_t({A_t^\prime}^*)^{-1}y_t & = & ({A_s^\prime}^*)^{-1}
(y_s \oplus S_s({A_t^\prime}^*)^{-1}y_t)\\
& = & ({A_{s+t}^\prime}^*)^{-1}((A_{s,t}^*)^{-1}(y_s \oplus S_sy_t))\\
& = & ({A_{s+t}^\prime}^*)^{-1}y_{s+t}.
\end{eqnarray*}  
In the above verification (and in the verification below) we have used our earlier observation that 
$T_t = ({A_t^\prime}^*)^{-1}(A_t^\prime)^{-1}S_t$ (see the proof of Proposition \ref{StTt},(c)) and the associativity of 
the sum system $(\{G^M_{a,b}\},\{S_t\})$. 
Conversely if $\{d_t\}$ is an imaginary cocycle for $(\{S_t\},\{T_t\})$ then we have 
$$\langle {A_t^\prime}^*d_t, 0 \oplus S_tx\rangle = \langle d_t, S_t x\rangle = 0, ~~~ \forall ~ x \in L^2(0,\infty),$$ 
and hence ${A_t^\prime}^*d_t\in G^M_{0,t}.$ Moreover
\begin{eqnarray*}
(A_{s,t}^*)^{-1}({A_s^\prime}^*d_s \oplus S_s{A_t^\prime}^*d_t) & = & {A_{s+t}^\prime}^*({A_{s+t}^\prime}^*)^{-1}(A_{s,t}^*)^{-1}({A_s^\prime}^*d_s \oplus S_s{A_t^\prime}^*d_t)\\
& = & {A_{s+t}^\prime}^*({A_s^\prime}^*)^{-1}\left({A_s^\prime}^*d_s \oplus S_s d_t\right)\\
& = & {A_{s+t}^\prime}^*(d_s + T_sd_t)\\
& = & {A_{s+t}^\prime}^* d_{s+t}.
\end{eqnarray*}
Hence an imaginary additive cocycle $\{d_t\}_{t>0}$ for the pair $(\{S_t\}, \{T_t\})$ is given by an imaginary addit $\{y_t\}_{t>0}$ of the corresponding sum system 
$(\{G^M_{a,b}\},\{S_t\})$ through the bijective correspondence $d_t = ({A_t^\prime}^*)^{-1} y_t$ and vice versa. (The above verification can as well be done without assuming $\{S_t\}$ being semigroup of isometries, by replacing $S_t$ by $(S_t^*)^{-1}.$)
\end{remark}

\medskip

The following lemma is probably well-known. 
We include a proof here for the reader's convenience.

\begin{lemma}\label{contiuous} Let $\{c_t\}$ be a real additive cocycle. 
Then $t \mapsto c_t$ is continuous. 
\end{lemma}

\begin{proof} We regard $\Hi_\R$ as a subspace of $L^2(\R)$ in a natural way. 
Let $\{U_t\}$ be the shift of $L^2(\R)$. 
By setting $c_{t}=-U_tc_{-t}$ for negative $t$, we can extend $c$ to an $\R$-cocycle. 
Let $V(t)=W(c_t)\EXP(U(t))\in \B(\Gamma(L^2(\R)))$. 
Then $\{V(t)\}$ is a measurable unitary representation of $\R$. 
It is well-known that such a representation is in fact continuous and so $t\mapsto W(c_t)$ is continuous in the strong 
operator topology. 
Thus $t\mapsto \inpr{W(c_t)\Phi}{\Phi}=e^{-\|c_t\|/2}$ is continuous, where $\Phi$ is the vacuum vector. 
Let $P$ be the projection from $\Gamma(L^2(\R))$ onto the one particle subspace of $\Gamma(L^2(\R))$. 
Then we have $e^{\|c_t\|/2}PW(c_t)\Phi=c_t$ and $c_t$ is continuous. 
\end{proof}

\medskip

Let $\{c_t\}_{t>0}$ be a real additive cocycle. 
Then Arveson's theorem \cite[Theorem 5.3.2]{Arv} shows 
that there exists $c\in L^2_{\mathrm{loc}}[0,\infty)$ such that 
$c_t=c-S_tc$, where we extend the shift $S_t$ to 
$L^2_{\mathrm{loc}}[0,\infty)$ in an obvious way. 
Lemma \ref{contiuous} and the cocycle relation imply that there exist positive constants $a,b$ 
such that $||c_t||\leq a+bt$, and so
$$\int_0^t|c(x)|^2dx\leq ||c_t||^2\leq (a+bt)^2.$$
Note that for every $\varepsilon>0$, 
\begin{eqnarray*}\int_0^\infty |c(x)|^2e^{-\varepsilon x}dx
&=&\int_0^\infty e^{-\varepsilon x}\frac{d}{dx}
\big(\int_0^x|c(y)|^2dy\big)dx\\
&=&\varepsilon\int_0^\infty e^{-\varepsilon x}\int_0^x|c(y)|^2dydx<\infty.
\end{eqnarray*}
Thus for $w\in \rH$, we have 
$$0=(c_t,T_te_w)=(c-S_tc,S_te_w+K_te_w)
=\int_0^\infty c(x+t)e^{-xw}dx-\Lt c(w)+(c,K_te_w).$$
Performing the Laplace transformation in the variable $t$, we get 
the following from \cite[Lemma 3.1]{I} for $z\in \rH$ with sufficiently 
large $\RE z$: 
$$\frac{\Lt c(w)-\Lt c(z)}{z-w}-\frac{\Lt c(w)}{z}+(c,e_z)
(\xi_{M,z},e_w)=0,$$ where $\xi_{M,z}$ is as defined in the beginning of \cite[Section 3]{I}.
Now \cite[Lemma 3.3]{I} implies 
$$\frac{z\Lt c(z)}{M(z)}=\frac{w\Lt c(w)}{M(w)},$$
and so $\Lt c(z)$ is proportional to $M(z)/z$. 

Thanks to the fact that $M(z)/(1+z)\in \Ha$, indeed there exists 
a function $c^M$ such that $c^Me_a\in \Hi$ for all $a>0$, and 
$\Lt {c^M}(z)=M(z)/z$. 
Note that $c^M$ is a real function.
We set $c^M_t=c^M-S_tc^M$. Then 
$$\Lt{c^M_t}=\frac{(1-e^{-tz})M(z)}{z},$$
which belongs to $\Ha$. 
Thus $c^M_t\in \LR{\infty}$.  
Tracing back the above argument, we can actually show the following Lemma.

\begin{lemma} Let $c^M$ be a function in $L^2_{\mathrm{loc}}[0,\infty)$ such that 
$\Lt c^M(z)=M(z)/z$ and let $c^M_t=c^M-S_tc^M$. 
Then $\{c^M_t\}_{t>0}$ is a real additive cocycle. 
Every real additive cocycle is a scalar multiple of $\{c^M_t\}_{t>0}$. 
\end{lemma}

\medskip

Let $\{d_t\}$ be an imaginary additive cocycle now. 
The condition $S_t^*d_t=0$ is equivalent to that the support of $d_t$ 
is in $(0,t]$. 
The cocycle relation 
$$d_{s+t}(x)=d_s(x)+(K_sd_t)(x)+(S_sd_t)(x)$$
shows that for every $s<x\leq t$, we have $d_{s+t}(x)=d_t(x-s)$ 
or equivalently, for every $0<x\leq t$, we have 
$d_{s+t}(s+t-x)=d_t(t-x)$. 
This means that there exists a function $d\in L^2_{\mathrm{loc}}[0,\infty)$ 
such that $d_t(x)=1_{(0,t]}(x)d(t-x)$. 
The cocycle relation is now equivalent to 
$$d(t+x)=d(x)+\int_0^tk(x,y)d(t-y)dy,\quad t,x>0.$$ 
Note that $t\mapsto \|d_t\|$ is a non-decreasing function. 
Thanks to the cocycle relation, we can see that 
$$t\mapsto \|d_t\|^2=\int_0^t|d(s)|^2ds$$
grows at most exponentially and the Laplace transformation $\Lt d(z)$ of 
$d$ exists for $z$ with sufficiently large $\RE z$. 
Performing the Laplace transformation of the above equation for the two 
variables $t$ and $x$ we get 
$$\frac{\Lt d(w)-\Lt d(z)}{z-w}=\frac{\Lt d(z)}{w}+\Lt{\xi_{M,z}}(w)\Lt d(w),$$
which is equivalent to 
$$zM(z)\Lt d(z)=wM(w)\Lt d(w).$$
Thus $\Lt d(z)$ should be proportional to $1/(zM(z))$. 
Indeed, \cite[Corollary 4.3]{I} shows that there exists 
$d^M\in L^2_{\mathrm{loc}}[0,\infty)$ such that $\Lt {d^M}(z)=1/(zM(z))$. 

\begin{lemma} Let $d^M\in L^2_{\mathrm{loc}}[0,\infty)$ 
such that $\Lt {d^M}(z)=1/(zM(z))$ and we set $d^M_t(x)=1_{(0,t]}(x)d^M(t-x)$. 
Then $\{d^M_t\}_{t>0}$ is an imaginary additive cocycle. 
Every imaginary additive cocycle is a scalar multiple of $\{d^M_t\}_{t>0}$. 
\end{lemma} 

\medskip

\begin{remark}
We set $c^M_{s,t}=S_sc^M_{t-s}$ and $d^M_{s,t}=T_sd^M_{t-s}$ for $0<s<t$. 
Then due to Lemma \ref{h} (also for instance the relation $\langle c^M_{s+t}, d^M_{s+t}\rangle = \langle c^M_{s}, d^M_{s}\rangle+ \langle c^M_{t}, d^M_{t}\rangle$ can easily be verified), there exists a constant $C$ such that the 
$\langle c^M_{q,r},d^M_{s,t} \rangle =C|[q,r]\cap[s,t]|$, where $|.|$ is the Lebesgue measure. 
As we have 
$$\langle c^M_t,d^M_t\rangle  =\int_0^tc^M(x)d^M(t-x)dx=c^M*d^M(t),$$
and $\Lt{c^M*d^M}(z)=\Lt {c^M}(z)\Lt {d^M}(z)=1/z^2$, we 
actually get $C=1$ in our case.  
A similar argument implies that the linear span of $\{d^M_t\}_{t>0}$ 
is dense in $\HiR$. 
Indeed, if $f\in \HiR$ is orthogonal to $d^M_t$ for every $t\geq 0$, 
$$0=\langle d^M_t,f \rangle =\int_0^td^M(t-x)f(x)dx=d^M*f(t),$$
and so 
$$0=\Lt{d^M*f}(z)=\Lt {d^M}(z)\Lt f(z)=\frac{\Lt f(z)}{zM(z)},$$ 
for $z\in \rH$ with sufficiently large $\RE z$.  
This implies $\Lt f(z)=0$ and $f=0$. 
\end{remark}

\medskip

We will show that the linear span of $\{c^M_{r,s}\}_{0<r<s\leq t}$ is dense in 
$G^M_{0,t}$ and the linear span of $\{d^M_{r,s}\}_{0<r<s\leq t}$ is dense in $L^2(0,t)$. 

\begin{lemma}\label{cspan} Let $M_I(z)$ be the inner component of $M(z)$ 
$($see \cite[Theorem 4.5]{I}$)$. 
Then the closed linear span of $\{c^M_{s,t}\}_{0<s<t<\infty}$ is 
$$\{f\in \Hi;\; \Lt f\in M_I\Ha\}.$$
\end{lemma} 

\begin{proof} Note that the linear span of $\{c_{s,t}\}_{0<s<t<\infty}$ is an 
invariant subspace of the shift $\{S_t\}_{t>0}$ and 
$$\Lt{c^M_{s,t}}(z)=\frac{(e^{-sz}-e^{-tz})M(z)}{z}.$$
Therefore, the statement follows from the Beurling-Lax theorem 
\cite[page 107]{H}. 
\end{proof}

\medskip

%\begin{lemma}\label{Gtperp} For $t>0$, $G^M(0,t)=\{f-S_tK_t^*f;\; f\in \LR t\}$. 
%\end{lemma}

%\begin{proof} 
%Since $\Hi_\R=L^2(0,t)_\R\oplus \Ran(S_t)$, Lemma\ref{direct sum},(ii) implies $G^M(0,t)=(1-S_tT_t^*)L^2(0,t)_\R$. 
%The statement follows from  $T_t^*f=K_t^*f$ for $f\in L^2(0,t)$.  
%\end{proof}

Recall that there exists a measurable function $k(x,y)$ such that 
$$K_tf(x)=1_{(0,t)}(x)\int_0^\infty k(t-x,y)f(y)dy,\quad f\in \Hi,$$ 
and there exists $a>0$ with 
$$\int_0^\infty\int_0^\infty |e^{-ax}k(x,y)|^2dxdy<\infty.$$  
\cite[Equation 3.1, Lemma 4.1]{I} shows 
$$\int_0^\infty\int_0^\infty k(x,y)e^{-xz-yw}dxdy=\frac{M(z)-M(w)}{M(z)(z-w)},$$
for $z,w\in \rH$ with $\RE z>a/2$. 

\begin{lemma}\label{G0inftyperp} A function $f\in \HiR$ belongs to ${G^M_{0,\infty}}^\perp$ 
if and only if there exists a positive number $a$ such that 
$$\int_{-\infty}^\infty \frac{\Lt f(-i\lambda)M(i\lambda)}{(z+i\lambda)
M(w+i\lambda)}d\lambda=0$$
holds for all $z,w\in \rH$ with $\RE w>a$.  
\end{lemma}

\begin{proof}
To make sense of the statement, first we claim that for $z,w\in \rH$ with $\RE w>a$, the function 
$$i\lambda \mapsto \frac{M(i\lambda)}{(z+i\lambda)M(w+i\lambda)}$$
belongs to $\Ha$. 
Let $h(x)=\int_0^xe^{-sw}k(s,x-s)ds$. 
Then since 
\begin{eqnarray*}
|h(x)|^2&\leq& \int_0^x|e^{-as}k(s,x-s)|^2ds\times \int_0^xe^{-2\RE (w-a)u}du\\
&\leq& \frac{1}{2\RE(w-a)}\int_0^x|e^{-as}k(s,x-s)|^2ds,
\end{eqnarray*}
\begin{eqnarray*}
\int_0^\infty dx\int_0^x ds|e^{-as}k(s,x-s)|^2
&=&\int_0^\infty ds\int_s^\infty dx|e^{-as}k(s,x-s)|^2\\
&=&\int_0^\infty ds\int_0^\infty du|e^{-as}k(s,u)|^2<\infty,
\end{eqnarray*}
the function $h$ belongs to $\Hi$. 
In consequence, we have  $e_z*h\in L^2(0,\infty)$.
On the other hand, for $\zeta\in \rH$ we have 
\begin{eqnarray*}
\Lt{e_z*h}(\zeta) &=&\frac{1}{z+\zeta}\int_0^\infty dxe^{-x\zeta}\int_0^x ds e^{-sw}k(s,x-s) \\
 &=&\frac{1}{z+\zeta}\int_0^\infty ds e^{-sw}\int_s^\infty dx e^{-x\zeta} k(s,x-s) \\
 &=&\frac{1}{z+\zeta}\int_0^\infty ds e^{-s(w+\zeta)}\int_0^\infty dy e^{-y\zeta} k(s,y)\\
 &=&\frac{1}{w(z+\zeta)}(1-\frac{M(\zeta)}{M(w+\zeta)}).
\end{eqnarray*}
As the function $i\lambda \mapsto 1/(z+i\lambda)$ belongs $\Ha$, we get the claim. 

If we interchange $S_t$ and $T_t$ in Lemma \ref{direct sum},(ii), we get that the orthogonal complement of 
$G^M_{0,t}$, which is $\Ran(T_t)$, is same as the range of the idempotent $T_tS_t^*$. 
Hence we conclude that a function $f \in L^2(0,\infty)$ belongs to the orthogonal complement of $G^M_{0,t}$ if and only if 
$$f=T_tS_t^*f= (K_tS_t^*+S_tS_t^*)f,\quad t>0.$$
This is equivalent to \begin{equation}\label{fk}
f(x)=\int_0^\infty k(s,y)f(s+x+y)dy,\quad x,s>0.\end{equation}
The Plancherel theorem implies that (\ref{fk}) is equivalent to 
\begin{equation}\label{Lfk}
f(x)=\frac{1}{2\pi}\int_{-\infty}^\infty \Lt{k(s,\cdot)}(i\lambda)
\Lt f(-i\lambda)e^{-i\lambda(s+x)}d\lambda.
\end{equation}
Let $z,w\in \rH$ with $\RE w>a$.  
Then via the Laplace transformation, (\ref{Lfk}) is equivalent to 
\begin{equation}\label{Lfm} \Lt f(z)=\frac{1}{2\pi}\int_{-\infty}^{\infty}\frac{\Lt f(-i\lambda)}{z+i\lambda}
\big(1-\frac{M(i\lambda)}{M(w+i\lambda)}\big)d\lambda.
\end{equation}
Note that since $\Lt f\in \Ha$, we have 
$$\Lt f(z)=\frac{1}{2\pi}\int_{-\infty}^{\infty}\frac{\Lt f(i\lambda)}{z-i\lambda}d\lambda.$$
Thus (\ref{Lfm}) is equivalent to 
$$\int_{-\infty}^{\infty}\frac{\Lt f(-i\lambda)}{z+i\lambda}
\frac{M(i\lambda)}{M(w+i\lambda)}d\lambda=0.$$
\end{proof}

\medskip

\begin{theorem}\label{ctdtspan} Let the notation be as above. Then 
\begin{itemize}
\item [$(1)$]
The linear span of $\{c^M_{r,s}\}_{0<r,s<t}$ is dense in 
$G^M_{0,t}$. 
\item [$(2)$]
The linear span of $\{d^M_{r,s}\}_{0<r,s<t}$ is dense in 
$L^2(0,t)_\R$. 
\end{itemize}
\end{theorem}

\begin{proof} (1) Let $G^{M,c}_{0,t}$ be the closure of the linear span of 
$\{c^M_{r,s}\}_{0<r<s<t}$, which is a closed subspace of $G^M_{0,t}$. 
First we show that $G^M_{0,\infty}=G^{M,c}_{0,\infty}$. 
Note that \cite[Corollary 4.3, Theorem 4.5]{I} implies, for $z,w\in \rH$ 
with $\RE w$ sufficiently large, that the function 
$$i\lambda \mapsto \frac{1}{(z+i\lambda)M(w+i\lambda)}$$
is an outer function in $\Ha$. 
Thus Lemma \ref{cspan} and Lemma \ref{G0inftyperp} imply that
$${G^{M,c}_{0,\infty}}^\perp \subset {G^M_{0,\infty}}^\perp,$$ 
which shows that $G^M_{0,\infty}=G^{M,c}_{0,\infty}$.

Let $G^{M,c,0}_{0,\infty}$ be the linear span of $\{c^M_{r,s}\}_{0<r<s<\infty}$. 
Then $G^{M,c,0}_{0,\infty}$ is invariant under the (not necessarily orthogonal) 
idempotent $1-S_tT_t^*$ from $L^2(0,\infty)$ onto $G^M_{0,t}$ thanks to Lemma \ref{direct sum}. 
Now the statement for finite $t$ follows from Lemma \ref{cspan} and from the fact that $G^M_{0,\infty}=G^{M,c}_{0,\infty}$. 

(2) We have already seen that the linear span of $\{d_{r,s}\}_{s>r>0}$ is dense in $\Hi_\R$. 
A similar argument works using the fact that 
$$L^2(0,t)_\R+ T_t\Hi_\R=\Hi_\R$$ 
is a topological direct sum. 
\end{proof}

\medskip

\begin{cor}\label{index 1} The sum system $(\{G^M_{a,b}\},\{S_t\})$ is divisible and of index 1. 
\end{cor}

\begin{remark}Note that we have 
$$\inpr{c^M_{q,r}}{c^M_{s,t}}=\frac{1}{2\pi}\int_{-\infty}^\infty 
\frac{(e^{-iq\lambda}-e^{-ir\lambda})(e^{is\lambda}-e^{it\lambda})
|M(i\lambda)|^2}{\lambda^2}d\lambda,$$
which depends only on the outer component of $M(z)$. 
Since $\{c^M_{q,r}\}_{0<q<r}$ determines the sum system $(\{G^M_{a,b}\},\{S_t\})$, 
it is isomorphic to the sum system for the outer component of $M$. 
On the other hand, the cocycle conjugacy class of the $E_0$-semigroup arising from $\{e^{tA_M}\}_{t>0}$ is determined 
by the sum system for $M$, and so we conclude that the cocycle conjugacy class is determined by the outer component of $M$.  
Note that this is in consistence with the fact that $G^M_{0,\infty}=\HiR$ when $M$ is an outer function and 
with part (c) in Proposition \ref{StTt}.
\end{remark}

\bigskip

%%%%%%%%%%%%%%%%%%%%%%%%%%%%%%%%%%%%%%%%%%%%%%%%%%%%%%%%%%%%%%%
%%%%%%%%%%%%%%%%%%%%%%%%%%%%%%%%%%%%%%%%%%%%%%%%%%%%%%%%%%%%%%%
%%%%%%%%%%%%%%%%%%%%%%%%%%%%%%%%%%%%%%%%%%%%%%%%%%%%%%%%%%%%%%%
\section{Application of the type III criterion}

In what follows, we assume that $M\in \HD_{2,\R}$ is an outer function unless otherwise stated. 
Then we have $\HiR=G^M_{0,\infty}$ and the sum system $\{G^M_{0,t}\}_{t>0}$ 
coincides with that discussed in \cite[Section 6]{I} except that complex 
Hilbert spaces are considered in \cite{I}. 
The function $|M(i\lambda)|^2$ corresponds to the spectral density 
function considered in \cite{T2} (see \cite[Section 6]{I}). 

Let $J^M_{t,0}$ be a linear operator whose domain is the linear span 
of $\{c^M_{r,s}\}_{0<r<s\leq t}$ sending $c^M_{r,s}$ to $d^M_{r,s}$. 
When $J^M_{t,0}$ has a bounded extension, we denote it by 
$J^M_t$ considered as an operator in $\B(G^M_{0,t},\LR t)$. As a consequence of Remark \ref{dtyt}, we can restate 
Theorem \ref{unitless} as below. 

We choose addits $x=\{x_t\}\in \RAU$ and $y=\{y_t\}\in \IAU$ of the sum system $(\{G^M_{0,t}\}, \{S_t\})$ 
such that $x_t=c_t$ and $\inpr{x_t}{y_t}=t$. 
We set $J$ to be the linear map $J:\RAU\rightarrow \IAU$ determined by $Jx=y$. 
Then we have $b_G(x,Jx)=1$ where $b_G$ is the bilinear form of $\RAU\times \IAU$ defined in Section 4, and so 
$b_G(\cdot,J\cdot)$ gives an inner product of $\RAU$. 
Let $J_{t,0}$ be as defined in Section 4. 
It is not hard to verify that $({A^\prime_t}^*)^{-1}({}^ty^\prime_{r,s})=d^M_{r,s}$, for all $r,s \leq t$. 
Then the theorem below follows from Theorem \ref{unitless}, since  $$J^M_{t,0} = ({A^\prime_t}^*)^{-1}|_{G^M_{0,t}} 
J_{t,0} ~~~ \mbox{and}~~~ A^\prime_t \in \cS(G^M_{0,t} \oplus L^2(t,\infty), L^2(0,\infty)),$$ 
where $A^\prime_t$ is as defined in Remark \ref{dtyt}.

\begin{theorem}\label{unitlessM} Let $M\in \HD_{2,\R}$ be an outer function. 
Then the generalized CCR flow arising from $\{e^{tA_M}\}_{t>0}$ 
is of type I if and only if the bounded extension $J^{M}_t$ exists 
and it is invertible such that for some positive number $a$, the operator  
$(J^{M}_t)^*J^{M}_t-a I$ is in the Hilbert-Schmidt class for all 
(some) $t>0$. 
\end{theorem}

\medskip

For a measurable (not necessarily bounded) function $F$ on the imaginary axis, 
the Toeplitz operator $\cT_F$ on $\Hi$ with domain $D(\cT_F)$ is defined 
as follows: 
$$D(\cT_F)=\{f\in \Hi;\; \int_{-\infty}^\infty |F(i\lambda)|^2
|\Lt f(i\lambda)|^2d\lambda<\infty\},$$ 
$$(\cT_Ff)(x)=1_{(0,\infty)}(x)\frac{1}{2\pi}\int_{-\infty}^\infty 
F(i\lambda)\Lt f(i\lambda)e^{ix\lambda}d\lambda,$$
where the above integral should be appropriately interpreted as usual.  

Assume that there exists a bounded operator $J^M\in \B(\Hi)$ whose restriction 
to $G^M_{0,t}$ is $J^M_t$. 
Then by construction, we have $J^MS_t=T_tJ^M$ and so $S_t^*J^MS_t=J^M$. 
Such an operator must be a Toeplitz operator (see, for example, 
\cite[Chapter13.3]{Arv}). 
Indeed, we can determine the symbol even in the case where the global 
extension $J^M$ does not exist.   

\begin{lemma}\label{Jtoeplitz} If $1/\{(1+z)M(z)\}\in \Ha$, then $c^M_{s,t}$ is in the domain 
of $\cT_{1/|M|^2}$ for $0<s<t$ and 
$$\cT_{\frac{1}{|M|^2}}c^M_{s,t}=d^M_{s,t}.$$
\end{lemma}

\begin{proof} The first statement is obvious. 
To prove the second statement, it suffices to show 
$\langle \cT_{1/|M|^2}c^M_{s,t},c^M_{q,r}\rangle =\langle d^M_{s,t},c^M_{q,r}\rangle $ for all 
$0<q<r$. 
Indeed, 
\begin{eqnarray*}\langle \cT_{\frac{1}{|M|^2}}c^M_{s,t},c^M_{q,r}\rangle &=&
\frac{1}{2\pi}\int_{-\infty}^\infty\frac{1}{|M(i\lambda)|^2}
\frac{(e^{-is\lambda}-e^{-it\lambda})M(i\lambda)}{i\lambda}
\overline{\frac{(e^{-iq\lambda}-e^{-ir\lambda})M(i\lambda)}{i\lambda}}
d\lambda\\
&=&\frac{1}{2\pi}\int_{-\infty}^\infty
\frac{(e^{-is\lambda}-e^{-it\lambda})(e^{iq\lambda}-e^{ir\lambda})}{\lambda^2}
d\lambda\\
&=&\langle 1_{(s,t]},1_{(q,r]}\rangle \\
&=&\langle d^M_{s,t},c^M_{q,r}\rangle . 
\end{eqnarray*}
\end{proof}

\medskip

\begin{remark} The condition $1/\{(1+z)M(z)\}\in \Ha$ in the above lemma is automatically satisfied for 
the class of $\{e^{tA_M}\}$ coming from Tsirelson's off white noises. 
An off white noise is a generalized Gaussian process with the spectral density functions $e^{\rho(\lambda)}$ 
(denoted by $W(\lambda)$ in \cite{T2}) satisfying the following two conditions (in fact, the first one automatically follows from 
the second one as we will see below): 
\begin{equation}\label{T1}\int_{-\infty}^\infty\frac{e^{\rho(\lambda)}}{1+\lambda^2}<\infty,\end{equation}
\begin{equation}\label{T2}\int_{-\infty}^\infty\int_{-\infty}^\infty \frac{|\rho(\lambda_1)-\rho(\lambda_2)|^2}
{|\lambda_1-\lambda_2|^2}d\lambda_1d\lambda_2<\infty.\end{equation} 
As described in \cite[Section 6]{I}, Tsirelson's $E_0$-semigroup arising from the spectral density function $e^{\rho(\lambda)}$ is 
conjugate to the generalized CCR flow given by an outer function $M(z)$ with the relation $|M(i\lambda)|^2=e^{\rho(\lambda)}$. 
Indeed, given a real function $\rho(\lambda)$ satisfying Equation (\ref{T2}), the function $M(z)\in \HD_{2,\R}$ is obtained by 
the integral 
$$M(z)=\exp\{\frac{1}{2\pi}\int_{-\infty}^\infty \frac{\lambda z+i}{\lambda+iz}\frac{\rho(\lambda)d\lambda}{1+\lambda^2}\}.$$
Via the change of variable $\zeta=(z-1)/(z+1)$,  the condition (\ref{T1}) is equivalent to that $N(\zeta)=M(z)$ belongs to 
the Hardy space $H^2(\D)$, where $\D$ is the unit disc, while the condition (\ref{T2}) is equivalent to that 
$\log N(\zeta)$ belongs to the analytic Besov space $B_2(\D)$ (see \cite{T2} and \cite[Chapter 5]{Z}). 
Since we have the inclusion relation $B_2(\D)\subset \mathrm{VMOA}(\D)$ (see \cite[Lemma 9.4.2]{Z}), 
the condition (\ref{T2}) solely implies both (\ref{T1}) and $1/N(\zeta)\in H^2(\D)$ thanks to \cite[page 100]{K}, 
and so $1/\{(1+z)M(z)\}\in\Ha$ holds.   
\end{remark}
\medskip

Tsirelson \cite[Lemma 10.2]{T1} shows that if 
$$\lim_{\lambda \to \pm\infty}M(i\lambda)=0,$$
the $E_0$-semigroup arising from $\{T_t\}_{t>0}$ is of type III. 
With Theorem \ref{unitlessM}, we can easily treat the other extreme case. 

\begin{theorem}\label{outerIII} Let $M\in \HD_{2,\R}$ be an outer function such that 
$1/\{(1+z)M(z)\}\in \Ha$. 
If  
$$\lim_{\lambda \to \pm\infty}|M(i\lambda)|=\infty,$$
then the generalized CCR flow arising from $\{T_t\}_{t>0}$ is of type III. 
\end{theorem}  

\begin{proof} We use the same function as in \cite[Lemma 10.1]{T1}. 
For a natural number $n$, we set 
$$f_n=\sum_{k=0}^{n-1}\big(c^M_{\frac{2k}{2n},\frac{2k+1}{2n}}
-c^M_{\frac{2k+1}{2n},\frac{2k+2}{2n}}\big)\in 
G^M_{0,1}.$$
Then $\Lt{f_n}(z)=F_n(z)M(z)$ with 
\begin{eqnarray*}
F_n(z)&=&\sum_{k=0}^{n-1}
\frac{e^{-\frac{2k+2}{2n}z}+e^{-\frac{2k}{2n}z}-2e^{-\frac{2k+1}{2n}z}}{z}
=\frac{(1-e^{-z})(1-e^{-\frac{z}{2n}})^2}{(1-e^{-\frac{z}{n}})z}\\
&=&\frac{(1-e^{-z})(1-e^{-\frac{z}{2n}})}{(1+e^{-\frac{z}{2n}})z}.
\end{eqnarray*}
The Plancherel theorem applied to the case with $M=1$ implies 
$$\int_{-\infty}^\infty |F_n(i\lambda)|^2 d\lambda=2\pi.$$
Suppose that $J^M_1$ exists and it is invertible. 
Then 
\begin{eqnarray*}
||f_n||^2&\leq& ||(J^M_1)^{-1}||^2||J^M_1f_n||^2
=||(J^M_1)^{-1}||^2||\cT_{\frac{1}{|M|^2}}f_n||\\
&\leq& \frac{||(J^M_1)^{-1}||^2}{2\pi}\int_{-\infty}^\infty
\frac{|F_n(i\lambda)|^2}{|M(i\lambda)|^2}d\lambda, 
\end{eqnarray*}
which would imply 
$$\int_{-\infty}^\infty |F_n(i\lambda)|^2|M(i\lambda)|^2d\lambda\leq 
||(J^M_1)^{-1}||^2 \int_{-\infty}^\infty
\frac{|F_n(i\lambda)|^2}{|M(i\lambda)|^2}d\lambda.$$
However, since the sequence $\{|F_n(z)|^2\}_{n=1}^\infty$ uniformly 
converges to 0 on every compact subset of the imaginary axis, 
we get a contradiction. 
\end{proof}

\medskip

One can also reproduce Tsirelson's result \cite[Lemma 10.2]{T1} using 
Theorem \ref{unitlessM}. 
Indeed, assume that $\lim_{\lambda\to\pm\infty}M(i\lambda)=0$ and $1/\{(1+z)M(z)\}\in \Ha$. 
If $\{T_t\}_{t\geq 0}$ gave a type I $E_0$-semigroup,  
a similar computation as above would show that 
$|(J^M_1f_n,f_n)|\leq ||J_1^M|| ||f_n||^2$ and 
$$2\pi=\int_{-\infty}^\infty |F_n(i\lambda)|^2d\lambda\leq ||J_1^M|| \int_{-\infty}^\infty |F_n(i\lambda)|^2 
|M(i\lambda)|^2d\lambda,$$
which is a contradiction. 

Among the spectral density functions treated in \cite{T2}, 
the ones to which the above Theorem \ref{outerIII} applies are as follows: 
strictly positive smooth functions such that for large $|\lambda|$, 
(1) $|M(i\lambda)|^2=\log^\beta|\lambda|$ for $\beta>0$, or 
(2) $|M(i\lambda)|^2=\exp(a\log^\beta{|\lambda|})$ with $a>0$ 
and $0<\beta<1/2$. 
It is quite likely that these families of spectral density functions give rise to mutually non-isomorphic 
product systems. 

\begin{problem} Show that the product systems arising from the above spectral density functions are 
mutually non-isomorphic. 
\end{problem}

Now we treat the case where $|M(i\lambda)|$ converges to a non-zero constant 
at infinity. 
Let $L^M_t$ be the restriction of $\cT_M$ to $D(\cT_M)\cap \LR t$. 
We claim that for $f\in D(\cT_M)$, 
the function $M(z)\Lt f(z)$ belongs to $\Ha$. 
Recall that a function $F(z)$ belongs to $\Ha$ if and only if 
$(1+z)F(z)$ belongs to the Hardy space $H^2(\D)$ of the unit disc 
via the change of variable $\zeta=(z-1)/(z+1)$. 
Since $M(z)/(1+z)$ and $\Lt f(z)$ are in $\Ha$ and 
$M(i\lambda)\Lt f(i\lambda)$ is square integrable, 
the function $(1+z)M(z)\Lt f(z)$ belongs to 
$H^1(\D)\cap L^2(\T)\subset H^2(\D)$ and we get the claim. 
As a consequence, we have $\cT_M(1_{(r,s]})=c^M_{r,s}$, and 
the operator $L^M_t$ is densely 
defined and its image is dense in $G^M_{0,t}$. 
We denote by $P_t$ the orthogonal projection from $\HiR$ onto $\LR t$. 
 
\begin{theorem}\label{sumsystemLM} Let $M_1,M_2$ be (not necessarily outer) functions in $\HD_{2,\R}$. 
We assume that $L^{M_i}_t$ is a bounded invertible operator from $\LR t$ onto 
$G^{M_i}_{0,t}$ for $i=1,2$.  
Then the sum system for $M_1$ and $M_2$ are isomorphic if and only if 
there exists a positive constant $a$ such that 
$${L^{M_1}_t}^*L^{M_1}_t-a{L^{M_2}_t}^*L^{M_2}_t$$ 
is a Hilbert-Schmidt operator for all (some) $0<t<\infty$. 
\end{theorem}

\begin{proof} Since the sum system for $M$ and ${L_t^M}^*L_t^M$ depend only on the outer component of $M$, 
we may and do assume that $M_1$ and $M_2$ are outer. 
Since $L^{M_1}_{t}(L^{M_2}_t)^{-1}c^{M_2}_t=c^{M_1}_t$, 
the two sum systems are isomorphic if and only if there exists a positive 
constant $a$ such that 
$$(L^{M_2}_t)^{-1*}(L^{M_1}_t)^*L^{M_1}_t(L^{M_2}_t)^{-1}-a1$$ 
is a Hilbert-Schmidt operator in $\B(G^{M_2}_{0,t})$ for all (some) $0<t<\infty$. 
Since $L^{M_2}_t$ is invertible, this is equivalent to that 
$(L^{M_1}_t)^*L^{M_1}_t-a(L^{M_2}_t)^*L^{M_2}_t$ is a Hilbert-Schmidt operator, 
which shows the statement. 
\end{proof}

\medskip

\begin{theorem}\label{tauMIII} Let $M\in \HD_{2,\R}$ be an outer function. 
We assume that there exist two positive constants $b_1$, $b_2$ such that 
$b_1\leq |M(i\lambda)|\leq b_2$ for almost every $\lambda$. 
Then the semigroup $\{e^{tA_M}\}_{t>0}$ gives rise to a type I 
$E_0$-semigroup if and only if there exists a positive constant 
$a$ such that 
$$P_t\cT_{|M|^2} P_t-aP_t\cT_{1/M}\cT_{1/M}^*P_t$$ 
is a Hilbert-Schmidt operator for all (some) $0<t<\infty$. 
If moreover $(I-P_t)\cT_{|M|^2}P_t$ is a Hilbert-Schmidt 
operator for all $0<t<\infty$, then this condition is equivalent to 
that $P_t\cT_{|M|^2-\sqrt{a}}P_t$ is a Hilbert-Schmidt for 
all $0<t<\infty$. 
\end{theorem} 

\begin{proof} By assumption, $M$ and $1/M$ are in $H^\infty(\rH)$ and 
$\cT_M$ and $\cT_{1/M}$ are bounded.  
We also have $\cT_{|M|^2}=\cT_M^*\cT_M$ and 
$\cT_{1/|M|^2}=\cT_{1/M}^*\cT_{1/M}$. 
Lemma \ref{Jtoeplitz} implies that $J^M_t$ is the restriction of 
$\cT_{1/M}^*\cT_{1/M}$ to $G^M_{0,t}$. 
Thus a similar argument as above using Theorem \ref{unitlessM} implies the first 
statement. 

Now assume that $(I-P_t)\cT_{|M|^2}P_t$ is a Hilbert-Schmidt operator for all $0<t<\infty$. 
Since 
$$(P_t\cT_{1/M}\cT_{1/M}^*P_t)(P_t\cT_{|M|^2}P_t)=P_t
-P_t\cT_{1/M}\cT_{1/M}^*(I-P_t)\cT_{|M|^2}P_t,$$
the operator  $P_t\cT_{|M|^2}P_t$ is the inverse of $P_t\cT_{1/M}\cT_{1/M}^*P_t$ modulo the 
Hilbert-Schmidt operators. 
Thus thanks to the first statement, the semigroup $\{e^{tA_M}\}_{t>0}$ gives rise to a type I 
$E_0$-semigroup if and only if there exists a positive constant 
$a$ such that 
$$(P_t\cT_{|M|^2} P_t)^2-aP_t$$ 
is a Hilbert-Schmidt operator for all $0<t<\infty$, which is equivalent that 
$P_t\cT_{|M|^2-\sqrt{a}I}P_t$ is a Hilbert-Schmidt operator for all $0<t<\infty$. 
\end{proof}

\medskip

Note that the assumptions of the above two theorems can be checked by using 
the Fourier transformation of the symbols in distribution sense. 

We now collect a few useful criteria for local boundedness of Toeplitz 
operators with unbounded symbols. 

For $f\in L^2(\R)$ we denote by $\hat{f}(\xi)$ its Fourier transformation 
with normalization 
$$\hat{f}(\xi)=\int_{-\infty}^\infty f(x)e^{-i\xi x}dx,\quad f\in 
L^1(\R)\cap L^2(\R).$$  
The next lemma may be regarded as a version of the well-known uncertainty principle. 

\begin{lemma}\label{uncertain} Let $E$ and $F$ be measurable subsets of $\R$ with finite 
Lebesgue measures $|E|$ and $|F|$ respectively. 
For $f\in L^2(\R)$ whose support is in $E$, we have 
$$\int_{F}|\hat{f}(\xi)|^2d\xi\leq |E||F|||f||^2.$$
\end{lemma}

\begin{proof} First we assume $f\in L^1(\R)\cap L^2(\R)$. 
We take $g\in L^2(\R)$ such that $\hat{g}$ is the characteristic 
function of $F$. 
Then the Plancherel theorem implies 
$$\int_F|\hat{f}(\xi)|^2d\xi=\int_{-\infty}^\infty|\hat{f}(\xi)\hat{g}(\xi)|^2d\xi=2\pi \int_{-\infty}^\infty |f*g(x)|^2dx.$$
On the other hand, 
$$|f*g(x)|^2=|\int_E f(t)g(x-t)dt|^2\leq ||f||^2\int_E|g(x-t)|^2dt,$$
and so  
\begin{eqnarray*}
\int_F|\hat{f}(\xi)|^2d\xi&\leq& 2\pi||f||^2\int_{-\infty}^\infty 
\int_E|g(x-t)|^2dtdx=2\pi ||f||^2||g||^2|E|\\
&=&||f||^2|E||F|.
\end{eqnarray*}
The statement in the general case follows from an easy approximation argument. 
\end{proof}

\medskip

\begin{lemma}\label{5.7} Let $M\in \HD_{2,\R}$. 
If there exists $\varepsilon>0$ such that 
the Lebesgue measure of 
$$\{\lambda\in \R;\; |M(i\lambda)|\leq \varepsilon\}$$
is finite, then there exists a positive constant $C_t$ for each 
$0<t<\infty$ such that 
$$||L^M_tf||\geq C_t||f||,\quad \forall f\in \LR t\cap D(\cT_M).$$ 
\end{lemma}

\begin{proof} Note that we have 
$$||\cT_Mf||^2=\frac{1}{2\pi}
\int_{-\infty}^\infty|M(i\lambda)|^2|\Lt f(i\lambda)|^2d\lambda,$$
for $f\in D(\cT_M)$ as we have $M(z)\Lt f(z)\in \Ha$. 
Let 
$$E_n=\{\lambda\in \R;\; |M(i\lambda)|\leq \frac{1}{n}\}.$$ 
By assumption, $|E_n|$ is finite for large $n$.  
Note that $M(i\lambda)\neq 0$ almost everywhere as $M(z)/(1+z)\in \Ha$, 
which implies that $\{|E_n|\}_{n=1}^\infty$ converges to 0. 
For $f\in \LR t\cap D(\cT_M)$, we have 
\begin{eqnarray*}
\int_{-\infty}^\infty|M(i\lambda)|^2|\Lt f(i\lambda)|^2d\lambda
&\geq& \int_{\R\setminus E_n}|M(i\lambda)|^2|\Lt f(i\lambda)|^2d\lambda\\
&\geq&\frac{1}{n}\int_{\R\setminus E_n}|\Lt f(i\lambda)|^2d\lambda\\
&=&\frac{2\pi ||f||^2-\int_{E_n}|\Lt f(i\lambda)|^2d\lambda}{n}\\
&\geq&\frac{(2\pi-t|E_n| )||f||^2}{n},
\end{eqnarray*}
where we use Lemma \ref{uncertain}.
Thus for sufficiently large $n$ with $2\pi>t|E_n|$, we get 
$$||\cT_Mf||^2\geq \frac{2\pi -t|E_n|}{2\pi n}||f||^2,$$
which finishes the proof. 
\end{proof}

\medskip

\begin{lemma}\label{5.8} Let $\varphi\in L^2(\R)$ and let $\Phi(i\lambda)=
\hat{\varphi}(\lambda)$. 
Then $\LR t\subset D(\cT_\Phi)$ and $\cT_\Phi P_t$ is a Hilbert-Schmidt 
operator for all $0<t<\infty$. 
\end{lemma}

\begin{proof} Let $f\in \LR t$. 
Since $\LR t\subset L^1(0,\infty)$, the convolution $\varphi*f$ 
makes sense as an element in $L^2(\R)$ and we get 
$\cT_\Phi f(x)=1_{(0,\infty)}(x)\varphi*f(x)$. 
In particular, $\cT_\Phi P_t$ is the integral operator with the kernel 
$1_{(0,\infty)}(x)\varphi(x-y)1_{(0,t]}(y)$ and so 
$$||\cT_\Phi P_t||_{\mathrm{H.S}}^2=\int_0^t\int_0^\infty|\varphi(x-y)|^2dxdy
=\int_{-t}^\infty|\varphi(u)|^2 \big(t\wedge (t+u)\big)du\leq 
t||\varphi||^2.$$
\end{proof}

\medskip

In a similar way as above, we can show the following: 
\begin{lemma}\label{5.9} Let $\varphi\in L^1(\R)$ and let $\Phi(i\lambda)
=\hat{\varphi}(\lambda)$. 
\begin{itemize}
\item [$(1)$] $\cT_\Phi$ is bounded. 
\item [$(2)$] $P_t\cT_\Phi P_t$ is a Hilbert-Schmidt operator for all 
$0<t<\infty$ if and only if $\varphi\in L^2_{\mathrm{loc}}(\R)$. 
\item [$(3)$] $(1-P_t)\cT_\Phi P_t$ is a Hilbert-Schmidt operator for all 
$0<t<\infty$ if and only if 
$$\int_0^\infty |\varphi(x)|^2(1\wedge x)dx<\infty.$$ 
\end{itemize}
\end{lemma}

\bigskip

%%%%%%%%%%%%%%%%%%%%%%%%%%%%%%%%%%%%%%%%%%%%%%%%%%%%%%%%%%%%%%%
%%%%%%%%%%%%%%%%%%%%%%%%%%%%%%%%%%%%%%%%%%%%%%%%%%%%%%%%%%%%%%%
%%%%%%%%%%%%%%%%%%%%%%%%%%%%%%%%%%%%%%%%%%%%%%%%%%%%%%%%%%%%%%%

\section{Examples}\label{examples} 
Let $\varphi$ be a real function 
in $L^1_{\mathrm{loc}}[0,\infty)\cap L^2((0,\infty),x\wedge 1dx)$ and 
let $M=1-\Lt \varphi$. 	
Then $M$ belongs to $\HD_{2,\R}$. 
We determine the type of the generalized CCR flow arising from such $M$. 

\begin{lemma}\label{LMbdd} Let the notation be as above. 
Then $L^M_t$ is a bounded invertible operator from 
$\LR t$ onto $G^M_{0,t}$ for all $0<t<\infty$. 
\end{lemma}

\begin{proof} Let $\varphi_1(x)=1_{(0,1]}(x)\varphi(x)$ and 
$\varphi_2(x)=1_{(1,\infty)}(x)\varphi(x)$. 
Then $\varphi_1\in L^1(0,\infty)$ and $\varphi_2\in \Hi$, and so  
Lemma \ref{5.7}, Lemma \ref{5.8}, and Lemma \ref{5.9} imply the statement. 
\end{proof}

\medskip

\begin{lemma}\label{taugn-gn} Let the notation be as above. 
For a natural number $n$, we set 
$$g_n(x)=\sum_{k=0}^{n-1}\big(1_{(\frac{2k}{2n},\frac{2k+1}{2n}]}(x)
-1_{(\frac{2k+1}{2n},\frac{2k+2}{2n}]}(x)\big).$$
Then $\{\cT_Mg_n-g_n\}_{n=1}^\infty$ converges to 0. 
\end{lemma}

\begin{proof}
Note that $\Lt{g_n}(z)=F_n(z)$ holds, where $F_n$ is as 
in the proof of Theorem \ref{outerIII}. 
Let $\Phi_i(z)=\Lt {\varphi_i}(z)$ for $i=1,2$, 
where $\varphi=\varphi_1+\varphi_2$ is the decomposition 
as in the proof of Lemma \ref{LMbdd}. 
Then $\cT_Mg_n-g_n=-\cT_{\Phi_1}g_n-\cT_{\Phi_2}g_n$. 
Since $\{g_n\}$ converges to 0 weakly, the second term converges to 0 
as $\cT_{\Phi_2}P_1$ is a compact operator thanks to Lemma \ref{5.8}. 
Since $\Phi_1(i\lambda)$ is a continuous function vanishing at infinity, 
we get 
$$||\cT_{\Phi_1}g_n||^2=\frac{1}{2\pi}\int_{-\infty}^\infty 
|\Phi_1(i\lambda)|^2
|F_n(i\lambda)|^2d\lambda\to 0,\quad (n\to\infty).$$
\end{proof}

\medskip

\begin{lemma}\label{6.3} Let $f,g,h\in L^1(0,1)\cap L^2((0,1),xdx)$. 
We regards $L^1(0,1)$ as a subspace of $L^1(\R)$ naturally 
and set $\tilde{g}(x)=g(-x)$. 
If $f,h\notin L^2(0,1)$ such that 
$$\limsup_{t\to+0}\frac{||h||_{L^2(t,1)}}{||f||_{L^2(t,1)}}=C<\infty,$$
then 
$$\lim_{t\to +0}\frac{
||h*\tilde{g}||_{L^2(t,1)}}{||f||_{L^2(t,1)}}=0. $$
\end{lemma}

\begin{proof} Let $0<t,\varepsilon <1$ with $t+\varepsilon <1$. 
Then
\begin{eqnarray*}||h*\tilde{g}||_{L^2(t,1)}^2
&\leq& \int_t^1dx\int_0^1dr\int_0^1ds|g(r)||g(s)||h(x+r)||h(x+s)|\\
&\leq&\int_t^1dx\int_0^1dr\int_0^1ds|g(r)||g(s)|
\frac{|h(x+r)|^2+|h(x+s)|^2}{2}\\
&=&||g||_1 \int_t^1dx\int_0^1dr|g(r)||h(x+r)|^2\\
&=&||g||_1\int_0^{1-t} |g(r)|||h||_{L^2(t+r,1)}^2 dr\\
&\leq&||g||_1\int_0^{\varepsilon} |g(r)|dr||h||_{L^2(t,1)}^2
+||g||_1^2||h||_{L^2(\varepsilon,1)}^2. 
\end{eqnarray*}
Thus 
$$\limsup_{t\to +0}\frac{
||h*\tilde{g}||_{L^2(t,1)}}{||f||_{L^2(t,1)}}\leq C||g||_1\int_0^{\varepsilon} |g(r)|dr,$$
and the statement holds.  
\end{proof}

\medskip

\begin{theorem}\label{phi1phi2} Let $\varphi_1,\varphi_2\in 
L^1_{\mathrm{loc}}[0,\infty)\cap L^2((0,\infty),x\wedge 1dx)_\R$ 
and let $M_i=1-\Lt {\varphi_i}$ for $i=1,2$.  
\begin{itemize}
\item [$(1)$] If $\varphi_1-\varphi_2\in \Hi$, the two sum systems for 
$M_1$ and $M_2$ are isomorphic. 
\item [$(2)$] Assume $\varphi_1\notin L^2(0,\infty)$. 
If the two sum systems for $M_1$ and $M_2$ are isomorphic, 
then the following two statements hold:
\begin{equation}\label{1phi}
\lim_{t\to+0}\frac{||\varphi_2||_{L^2(t,\infty)}}
{||\varphi_1||_{L^2(t,\infty)}}=1,
\end{equation}
\begin{equation}\label{2phi}
\lim_{t\to+0}\frac{||\varphi_1-\varphi_2||_{L^2(t,\infty)}}
{||\varphi_1||_{L^2(t,\infty)}}=0.
\end{equation}
\end{itemize}
\end{theorem}

\begin{proof} Lemma \ref{LMbdd} shows that $L^{M_i}_t$ is a bounded invertible 
operator from $\LR t$ onto $G^{M_i}_{0,t}$. 
Therefore  we can apply Theorem \ref{sumsystemLM} to prove the theorem.  

(1) Assume that $\varphi_1-\varphi_2\in \Hi$. 
Lemma \ref{5.8} implies that $L^{M_1}_t-L^{M_2}_t$ is a Hilbert-Schmidt operator 
for all $0<t<\infty$ and Theorem \ref{sumsystemLM} shows that the sum systems for 
$M_1$ and $M_2$ are isomorphic. 

(2) Assume now that the sum systems for $M_1$ and $M_2$ are isomorphic. 
By the result just proved above, we may add a function in $\Hi$ to 
$\varphi_i$ without changing the assumption. 
Therefore we may and do assume that $\varphi_i$ has support in $(0,1)$ 
and $||\varphi_i||_1<1$ by truncating $\varphi_i$. 
As a consequence, $\cT_{M_i}$ is a bounded operator and we have 
$(L^{M_i}_t)^*L^{M_i}_t=P_t\cT_{|M_i|^2}P_t$ now. 
Assume that $\varphi_1\notin L^2(0,1)$. 

Theorem \ref{sumsystemLM} implies that there exists a positive constant $a$ such that 
$P_t\cT_{|M_1|^2-a|M_2|^2}P_t$ is a Hilbert-Schmidt operator for all 
$0<t<\infty$. 
This is possible only if $a=1$ thanks to Lemma \ref{taugn-gn}. 
Let $\Phi_i=\Lt{\varphi_i}$. 
Then 
$$|M_1|^2-|M_2|^2=\Phi_2-\Phi_1+\overline{\Phi_2}-\overline{\Phi_1}+
|\Phi_1|^2-|\Phi_2|^2,$$ and we have 
$\cT_{|M_1|^2-|M_2|^2}f=1_{(0,\infty)}(x)g*f(x)$ for $f\in \Hi$ with 
$$g=\varphi_2-\varphi_1+\tilde{\varphi}_2-\tilde{\varphi}_1+
\tilde{\varphi}_1*\varphi_1-\tilde{\varphi}_2*\varphi_2\in L^1(\R),$$
where the notation in Lemma \ref{6.3} is used. 
Lemma \ref{5.9} implies that $g\in L^2_{\mathrm{loc}}(\R)$ or equivalently 
$g\in L^2([-1,1])$ as $g$ is supported by $[-1,1]$. 
Setting $h=\varphi_2-\varphi_1$, we have
$$g=h+\tilde{h}-h*\tilde{h}-\varphi_1*\tilde{h}-h*\tilde{\varphi}_1.$$
%If $h\in L^2(0,1)$ we have nothing to prove, and so we assume $h\notin L^2(0,1)$. 

First suppose that (\ref{1phi}) does not hold. 
Exchanging the roles of $\varphi_1$ and $\varphi_2$ if necessary, 
we may assume that there exist a number $0<\delta<1$ and 
a decreasing sequence of positive numbers $\{t_n\}_{n=1}^\infty$ 
converging to 0 such that 
$$||\varphi_2||_{L^2(t_n,1)}\leq (1-\delta)
||\varphi_1||_{L^2(t_n,1)}$$
holds for all $n\in \N$. 
Then $||h||_{L^2(t_n,1)}\geq \delta||\varphi_1||_{L^2(t_n,1)}$. 
On the other hand we have 
$$1=\frac{
||g+h*\tilde{h}+\varphi_1*\tilde{h}+h*\tilde{\varphi}_1||_{L^2(t_n,1)}
}{||h||_{L^2(t_n,1)}},$$
whose right-hand side would converge to 0 due to Lemma \ref{6.3}, 
which is a contradiction. 
Thus Equation (\ref{1phi}) holds. 

Equation (\ref{2phi}) can be shown in a similar way in the presence of (\ref{1phi}).
\end{proof}

\medskip

\begin{theorem}\label{unitlessphi} Let $\varphi \in L^1_{\mathrm{loc}}[0,\infty)\cap 
L^2((0,\infty),x\wedge 1dx)_\R$ and let $M=1-\Lt {\varphi}$. 
Then the generalized CCR flow arising from $\{e^{tA_M}\}_{t>0}$ 
is of type I if and only if $\varphi\in \Hi$. 
In consequence, if $\varphi\notin \Hi$, then the corresponding generalized CCR flow is of type III.  
\end{theorem}

\begin{proof} Thanks to Theorem \ref{phi1phi2}, (1), to prove the theorem 
we may assume that $\varphi$ is supported by $(0,1)$ and 
$||\varphi||_1<1$. 
In this case, the function $M(z)$ is continuous and 
$$1-||\varphi||_1\leq |M(z)|\leq 1+||\varphi||_1$$
holds on the closure of $\rH$. 
In particular $M$ is an outer function. 
Therefore we can apply Theorem \ref{tauMIII} to $M(z)$. 

First we claim that $(1-P_t)\cT_{|M|^2}P_t$ is a Hilbert-Schmidt operator 
for all $0<t<\infty$. 
Let $\Phi=\Lt \varphi$. 
Then 
$$(1-P_t)\cT_{|M|^2}P_t=-(1-P_t)\cT_{\Phi}P_t
-(1-P_t)\cT_{\overline{\Phi}}P_t+(1-P_t)\cT_{|\Phi|^2}P_t.$$
Thanks to Lemma \ref{5.9}, the first two terms above are Hilbert-Schmidt operators. 
Note that we have $\cT_{|\Phi|^2}f=1_{(0,\infty)}
\varphi*\tilde{\varphi}*f$ for $f\in \Hi$. 
Thus to prove the claim, it suffices to show that
$$\int_0^1 t|\varphi*\tilde{\varphi}(t)|^2dt<\infty.$$
Indeed, we have 
\begin{eqnarray*}
\int_0^1 x|\varphi*\tilde{\varphi}(x)|^2dx &\leq&\int_0^1xdx\int_0^1dr\int_0^1ds|\varphi(r)||\varphi(s)|
|\varphi(x+r)||\varphi(x+s)| \\
 &\leq& \int_0^1xdx\int_0^1dr\int_0^1ds|\varphi(r)||\varphi(s)|
\frac{|\varphi(x+r)|^2+|\varphi(x+s)|^2}{2} \\
 &=&||\varphi||_1 \int_0^1xdx\int_0^1ds|\varphi(s)||\varphi(x+s)|^2 \\
 &=&||\varphi||_1 \int_0^1ds|\varphi(s)|\int_s^1dy(y-s)|\varphi(y)|^2 \\
 &\leq &||\varphi||_1^2\int_0^1x|\varphi(x)|^2dx<\infty. 
\end{eqnarray*}

Theorem \ref{tauMIII} together with the above claim shows that the generalized CCR flow arising 
from $\{e^{tA_M}\}_{t>0}$ is of type I if and only if 
$P_t\cT_{|M|^2-1}P_t$ is a Hilbert-Schmidt operator for all $t$, 
which is further equivalent to that 
$$g=\varphi*\tilde{\varphi}-\varphi-\tilde{\varphi}$$
is in $L^2(-1,1)$ thanks to Lemma \ref{5.9},(2). 
Lemma \ref{6.3} shows that this is equivalent to the condition 
$\varphi\in L^2(0,1)$. 
\end{proof}

\medskip

\begin{remark}
Let $\varphi$ be as above and $||\varphi||_1<1$. 
Then as shown in \cite[Section 6]{I}, the sum system $G^M_{0,t}$ can be 
realized as follows: $G^M_{0,t}=\LR t$ as a topological vector space with 
a new inner product 
$$\inpr{f}{g}_G=\frac{1}{2\pi}\int_{-\infty}^\infty 
\hat{f}(\lambda)\overline{\hat{g}(\lambda)} 
|M(i\lambda)|^2d\lambda.$$
Note that the norm of $G^M_{0,t}$ is equivalent to the usual 
$L^2$-norm as we have 
$$1-||\varphi||_1\leq |M(i\lambda)|\leq 1+||\varphi||_1.$$ 
This means that the invariant for a product system (through the associated sum system) 
discussed in \cite{T1} and \cite{pdct} can not distinguish the product system  
corresponding to this $M$ from the exponential product system of index 1.
\end{remark}

\bigskip

\section{Type III factors and type III $E_0$-semigroups}
\medskip
\subsection{Local algebras}
For a product system $E=(E_t)_{t>0}$, following \cite{Vol}, \cite{T1} and \cite{pdct}, we introduce an analogue of the observable algebra for 
a finite interval $I=(s,t)\subset (0,a)$ by 
$$\cA^E(I)~=~ U^a_{s,t}~(\C1_{E_s} \otimes 
\B(E_{t-s})\otimes 
\C1_{E_{a-t}}) {U^a_{s,t}}^*,$$ where $U^a_{s,t}$ is the unique unitary 
operator 
between the Hilbert spaces $E_s 
\otimes E_{t-s} \otimes E_{a-t}$ and $H_a$, determined uniquely by the
associativity 
of the 
product system.

For a family $\{\cA_\lambda\}_{\lambda\in \Lambda}$ of von Neumann algebras, we use the notation 
$\bigvee_{\lambda\in \Lambda}\cA_\lambda$ for the von Neumann algebra generated by 
$\{\cA_\lambda\}_{\lambda\in \Lambda}$. 
For an open subset $U$ of $(0,a)$, we set 
$$\cA^E(U)=\bigvee_{I\subset U}\cA^E(I)$$
where $I$ runs over all intervals contained in $U$. 
Then the isomorphism class of $\cA^E(U)$ does not depend on the choice of $a$. 

%Since we have 
%$$\bigvee_{n\in \N}\cA^E((s+\frac{1}{n},t-\frac{1}{n}))=\cA^E((s,t)),$$
Note that we can always ignore a finite subset of $(0,\infty)$ whenever we deal with $\cA^E(U)$ as we have  
$\cA^E((r,s)\cup (s,t))=\cA^E((r,t))$ (see for instance Corollary 25 in \cite{pdct}). 

When $U$ is an elementary set, that is a finite union of intervals, then 
$\cA^E(U)$ is a type I factor because it is generated by finitely many mutually commuting type I factors. 
However, when $U$ has infinitely many components, it is not clear whether $\cA^E(U)$ is of type I or even 
it is not clear whether $\cA^E(U)$ is a factor. 
Indeed, Murray von Neumann's notion of the type of $\cA^E(U)$ gives a computable invariant of the product system 
$(E_t)_{t>0}$. 

\begin{lemma} Let $E=(E_t)_{t>0}$ be a product system and let $U$ be a bounded non-empty open subset of $(0,\infty)$. 
Then  
\begin{itemize}
\item [$(1)$] If $E$ has a unit, then $\cA^E(U)$ has a direct summand that is a type I$_\infty$ factor. 
\item [$(2)$] If $E$ is of type I, then $\cA^E(U)$ is a type I$_\infty$ factor. 
\end{itemize}
\end{lemma}

\begin{proof} (1) We choose $a>0$ such that $U\subset (0,a)$ and assume that $\cA^E(U)$ acts on $E_a$. 
Since $U$ is open, there exist mutually disjoint open intervals $I_n$ such that 
$U=\bigcup_{n=1}^\infty I_n$. 
Let $v=(v_t)_{t>0}$ be a unit. 
Then we may and do assume that $||v_t||=1$ for all $t>0$. 
Let $L=\overline{\cA^E(U)v_a}$ and $P_L$ be the projection from $E_a$ onto $L$, which belongs to 
$\cA^E(U)'$. 
We introduce a state $\omega$ of $\cA^E(U)$ by $\omega(x)=\inpr{xv_1}{v_1}$. 
Then for $x_i\in \cA^E(I_i)$, we have 
$$\omega(x_1x_2\cdots x_n)=\omega(x_1)\omega(x_2)\cdots \omega(x_n),$$
which shows that $\omega$ is a product pure state of $\bigotimes_{i=1}^n\cA(I_i)\subset \cA^E(U)$ for 
all $n$. Therefore $\cA^E(U)P_L$ is a type I$_\infty$ factor. 

(2) We may assume that $E$ is the exponential product system with the test function space $L^2((0,\infty),K)$, 
where $K$ is the multiplicity space. 
Then $\cA^E(U)$ is nothing but $\B(\Gamma(L^2(U,K)))$.  
\end{proof}

\medskip

\begin{remark} If $E$ has a unit and $U=\bigcup_{n=1}^\infty (a_n,b_n)$ with $b_n<a_{n+1}$, then 
it is easy to show that $\cA^E(U)$ is a type I factor. 
Indeed, using the notation of the proof of (1) above, we have 
$$\overline{\cA^E(U)'L}=E(0,a)$$
in this case. 
Therefore the defining representation of $\cA^E(U)$ is quasi-equivalent to its restriction to $L$. 
\end{remark}

\bigskip

\subsection{Type of $\cA^\varphi(U)$} 
We fix $\varphi\in L^1_{\mathrm{loc}}[0,\infty)\cap L^2((0,\infty),1\wedge xdx)_\R\setminus \Hi$ and set 
$\Phi=\Lt \varphi$, $M=1-\Phi$. 
Let $H=\Gamma(\Hi)$ and $\alpha^\varphi$ be the $E_0$-semigroup acting on $\B(H)$ defined by 
$$\alpha^\varphi_t(W(f+ig))=W(S_t+iT_tg),\quad f,g\in \Hi_\R.$$
Then $\alpha^\varphi$ is of type III as we saw in Theorem \ref{tauMIII}. 
Since the cocycle conjugacy class of $\alpha^\varphi$ does not change under a $L^2$-perturbation of $\varphi$, 
we assume that $\varphi$ is supported by $[0,1)$ and $||\varphi||_1<1/3$. 
Note that the Toeplitz operator $\cT_M$ is a bounded invertible operator with $\cT_M^{-1}=\cT_{1/M}$ in this case. 

Let $E^\varphi$ be the product system for $\alpha^\varphi$ and $\cA^{\varphi}(U):=\cA^{E^\varphi}(U)$. 
Then we may and do identify $\cA^\varphi((0,t))$ with $\alpha^\varphi_t(\B(H))'$ and $\cA^\varphi((s,t))$ with  
$\alpha_s^\varphi(\cA^\varphi((0,t-s)))$. 

\begin{lemma} Let the notation be as above. 
Then 
$$\cA^\varphi(U)=\{W(\cT_Mf+i\cT_{M}^{*-1}g);\; f,g\in L^2(U)_\R\}''. $$
In consequence, the von Neumann algebra $\cA^\varphi(U)$ is either a type I factor or a type III factor. 
\end{lemma}

\begin{proof} Let $I=(s,t)$ be an open interval. 
When $s=0$, using the duality theorem \cite[Theorem 1',(5)]{Ara1}, we have 
\begin{eqnarray*}
\cA^\varphi(0,t)&=&
\{W(S_tf+iT_tg);\; f,g\in \Hi_R\}'\\
&=&\{W(f+ig);\; f\in G^M_{0,t},\; g\in L^2(0,t)_\R\}''.
\end{eqnarray*}
Therefore
$$\cA^\varphi((s,t))=\{W(S_sf+iT_sg);\;f\in G^M_{0,t-s},\; g\in L^2(0,t-s)_\R \}'',$$
and thanks to Theorem \ref{ctdtspan}, 
$$\cA^\varphi((s,t))=\{W(\xi c_{x,y}+i\eta d_{u,v});\;s<x<y<t,\; s<u<v<t,\; \xi,\eta\in \R \}''.$$
Note that we have $c_{x,y}=\cT_M 1_{(x,y)}$ and 
$$d_{u,v}=\cT_{1/|M|^2}c_{u,v}=\cT_{1/M}^*\cT_{1/M}c_{u,v}=\cT_M^{*-1}1_{(s,t)},$$
and so we get  
$$\cA^\varphi(I)=\{W(\cT_Mf+i\cT_{M}^{*-1}g);\; f,g\in L^2(I)_\R\}''. $$

For an open subset $U$ of $(0,1)$, we take mutually disjoint open intervals $I_n$ such that 
$U=\bigcup_{n=1}^\infty I_n$. 
Then 
\begin{eqnarray*}
\cA^\varphi(U)&=&\bigvee_{n=1}^\infty\{W(\cT_Mf+i\cT_{M}^{*-1}g);\; f,g\in L^2(I_n)_\R\}'' \\
 &=&\{W(\cT_Mf+i\cT_{M}^{*-1}g);\; f,g\in L^2(U)_\R\}''. 
\end{eqnarray*}

Let $K_1=\cT_M L^2(U)_\R$ and $K_2=\cT_M^{-1*}L^2(U)_\R$, which are closed spaces of $\Hi_\R$. 
To show that $\cA^\varphi(U)$ is a factor, it suffices to show that $K_1\cap K_2^\perp=K_1^\perp\cap K_2=\{0\}$ 
(see \cite[Theorem 1',(4)]{Ara1}). 
This immediately follows from $\inpr{\cT_Mf}{\cT_M^{-1*}g}=\inpr{f}{g}$. 
The rest of the statement follows from \cite{Ara2}. 
\end{proof}

\medskip

Let $K_1$ and $K_2$ be as above. 
It is shown in \cite[Lemma 4.1]{Ara1}, \cite{Ara2} that there exists an unique closed operator 
$V:K_1\rightarrow K_1^\perp$ such that $K_2$ is the graph of $V$.  
Moreover, the factor $\cA^\varphi(U)$ is of type I (resp. type III) if and only if $V$ is (resp. is not) 
a Hilbert-Schmidt class operator. 
(Though $K_1\cap K_2=K_1^\perp\cap K_2^\perp=\{0\}$ is also assumed in \cite{Ara1}, 
it is not really necessary.)  

\begin{lemma}\label{tauphiIII} Let the notation be as above and let $P_U$ be the projection onto $L^2(U)_\R$. 
Then $V$ is the restriction of $-\cT_M^{-1*}(I-P_U)\cT_M^*$ to $K_1$. 
In consequence, the factor $\cA^\varphi(U)$ is of type I if and only if 
$(I-P_U)\cT_{\Phi+\overline{\Phi}-|\Phi|^2}P_U$ is a Hilbert-Schmidt class operator. 
\end{lemma}

\begin{proof} Let $V_0:=-\cT_M^{-1*}(I-P_U)\cT_M^*|_{K_1}$. 
For $f,g\in L^2(U)_\R$, we have 
$$\inpr{V_0\cT_Mf}{\cT_Mg}=-\inpr{(I-P_U)\cT_M^*\cT_Mf}{g}=0,$$
which shows that $V_0$ is a bounded operator from $K_1$ to $K_1^\perp$. 
Moreover, we have 
$$\cT_Mf+V_0\cT_Mf=\cT_Mf-\cT_M^{-1*}(I-P_U)\cT_M^*\cT_Mf=\cT_M^{-1*}P_U\cT_M^*\cT_Mf\in K_2.$$
Therefore to prove that $V=V_0$, it suffices to show that 
$P_U\cT_M^*\cT_MP_U$ is invertible in $\B(L^2(U))$. 
Since 
$$\cT_M^*\cT_M=I-\cT_\Phi-\cT_\Phi^*+\cT_\Phi^*\cT_\Phi,$$
this follows from $||\cT_\Phi||\leq ||\varphi||_1<1/3$. 

Since $\cT_MP_U$ is an invertible operator from $L^2(U)_\R$ onto $K_1$, we conclude that $V$ is a Hilbert-Schmidt class 
operator if and only if 
$$\cT_M^*V\cT_MP_U=-(I-P_U)\cT_{|M|^2}P_U=(I-P_U)\cT_{\Phi+\overline{\Phi}-|\Phi|^2}P_U$$ 
is a Hilbert-Schmidt class operator. 
\end{proof}

\medskip

For two subsets $E$, $F$ of $\R$, we denote by $E\ominus F$ the symmetric difference of $E$ and $F$. 

\begin{theorem}\label{typeIorIII} Let $\varphi\in L^1_{\mathrm{loc}}[0,\infty)\cap L^2((0,\infty),1\wedge xdx)_\R,$ and let 
$U$ be a bounded open subset of $(0,\infty)$. 
We regards $\varphi$ as an element of $L^1(\R)$ and set $\tilde{\varphi}(x)=\varphi(-x)$. 
Then $\cA^\varphi(U)$ is a type I factor (resp. type III factor) if and only if the integral 
$$\cI_1=\int_0^1|\varphi(x)-\varphi*\tilde{\varphi}(x)|^2|(U+x)\ominus U|dx$$
converges (resp. diverges). 
When there exists $a>0$ such that $\varphi(x)$ is monotone on $(0,a)$, the above integral converges if and only if 
the integral 
$$\cI_2=\int_0^1|\varphi(x)|^2|(U+x)\ominus U|dx$$
converges. 
\end{theorem}

\begin{proof} Note that the isomorphism class of $\cA^\varphi(U)$ and whether $\cI_1$ is finite or not  
are stable under $L^2$-perturbation of $\varphi$. 
Therefore we may and do assume that $\varphi$ is supported by $[0,1]$ and $||\varphi||_1<1/3$. 
Then thanks to Lemma \ref{tauphiIII}, the factor $\cA^E(U)$ is of type I if and only if 
$$\cI_3:=\int_{(0,\infty)\setminus U}dx\int_Udy|\varphi(x-y)+\tilde{\varphi}(x-y)-\varphi*\tilde{\varphi}(x-y)|^2
<\infty,$$
because $(I-P_U)\cT_{\Phi+\overline{\Phi}-|\Phi|^2}P_U$ is the integral operator with the kernel 
$$(\varphi(x-y)+\tilde{\varphi}(x-y)-\varphi*\tilde{\varphi}(x-y))1_{(0,\infty)\setminus U}(x)1_U(y).$$
On the other hand, we have 
\begin{eqnarray*}
\cI_3 &=&\int_{-1}^1|\varphi(s)+\tilde{\varphi}(s)-\varphi*\tilde{\varphi}(s)|^2|([0,\infty)\setminus U)\cap (U+s) |ds \\
 &=& \int_0^1|\varphi(s)-\varphi*\tilde{\varphi}(s)|^2
 \big(|([0,\infty)\setminus U)\cap (U+s) |+|([0,\infty)\setminus U)\cap (U-s)|\big)ds\\
 &=&\int_0^1|\varphi(s)-\varphi*\tilde{\varphi}(s)|^2
 (|(U+s)\ominus U|-|[0,s]\cap U|)ds.\\
\end{eqnarray*}
Since 
$$\int_0^1|\varphi(s)-\varphi*\tilde{\varphi}(s)|^2|[0,s]\cap U|ds
\leq \int_0^1|\varphi(s)-\varphi*\tilde{\varphi}(s)|^2sds<\infty,$$
we get the first statement. 

Assume now that $\varphi$ is monotone on $(0,a)$. 
When $\lim_{x\to +0}\varphi(x)$ is finite, the function $\varphi$ is in $L^1(0,\infty)\cap \Hi$ and 
$\cI_1$ and $\cI_2$ are finite. 
Assume that the limit at 0 does not exist. 
Then by truncating $\varphi$ if necessary, we may assume that either $\varphi(x)$ is non-negative and decreasing 
or $\varphi(x)$ is non-positive and increasing. 
Assume that $\varphi(x)$ is non-negative and decreasing. 
(The other case can be treated in a similar way.) 
Then for $0<x<1$ we have 
\begin{eqnarray*}
\varphi(x)&\geq& \varphi(x)-\varphi*\tilde{\varphi}(x)=\varphi(x)-\int_0^1\varphi(x+t)\varphi(t)dt\\
&\geq& \varphi(x)-\varphi(x)\int_0^1\varphi(t)dt
=(1-||\varphi||_1)\varphi(x). 
\end{eqnarray*}
Therefore $\cI_1$ is finite if and only if $\cI_2$ is finite. 
\end{proof}

\medskip

When $U$ is an elementary set, we have $|(U+x)\ominus U|=O(x)$, $(x\to 0)$ and $\cI_1$ is always finite, 
which recovers the fact that $\cA^\varphi(U)$ is a type I factor. 
To construct $U$ such that $\cA^\varphi(U)$ is of type III, we have to control the speed of 
convergence of $\lim_{x\to +0}|(U+x)\ominus U|=0$.

\begin{lemma}\label{estimate} Let $f(x)$ be a non-negative strictly decreasing continuous function on 
$(0,\infty)$ satisfying 
$\lim_{x\to +0}f(x)=\infty$, $\lim_{x\to \infty}f(x)=0$, and 
$$\int_0^1f(x)dx<\infty.$$ 
We set $a_0=0$, $a_{2n-1}=a_{2n}=f^{-1}(n)$ for $n\in \N$, and $b_n=\sum_{k=0}^na_k$. 
Let $I_n=(b_{2n},b_{2n+1})$ and $U=\bigcup_{n=0}^\infty I_n$. 
Then $U$ is a bounded open set and the following estimate holds for every $0<x<f^{-1}(1)$:
$$(2f(x)-1)x\leq |(U+x)\ominus U|\leq (2f(x)+1)x+2\int_0^xf(s)ds.$$ 
\end{lemma}

\begin{proof} Note that we have 
$$\sum_{k=n+1}^\infty f^{-1}(k)\leq \int_0^{f^{-1}(n+1)}f(s)ds-nf^{-1}(n+1)$$
and in particular, the set $U$ is bounded. 
For $0<x\leq f^{-1}(1)$, we choose $n$ so that $f^{-1}(n+1)<x\leq f^{-1}(n)$, 
which is equivalent to $n\leq f(x)<n+1$. 
Then we have 
$$nx\leq |(U+x)\setminus U|\leq nx+\sum_{k=n+1}^\infty|I_k|,$$
$$(n+1)x\leq |(U-x)\setminus U|\leq (n+1)x+\sum_{k=n+2}^\infty|I_k|,$$
and so 
$$(2n+1)x\leq |(U+x)\ominus U|\leq (2n+1)x +2\sum_{k=n+1}^\infty f^{-1}(k).$$
Since $f(x)-1< n$, we get the lower estimate. 
Since 
$$\sum_{k=n+1}^\infty f^{-1}(k)\leq \int_0^{f^{-1}(n+1)}f(s)ds-nf^{-1}(n+1)\leq 
\int_0^xf(s)ds$$
and $n\leq f(x)$, we get the upper estimate. 
\end{proof}

\medskip

For two real valued functions $h(x)$ and $g(x)$ on $(0,\infty)$, we denote 
$h(x)\asymp g(x)$ if there exist positive numbers $a,b,c>0$ such that  
$$bh(x)\leq g(x)\leq ch(x),\quad \forall x\in (0,a).$$ 

\begin{cor}\label{symmdiff} 
Assume that $f(x)=x^{\gamma-1}L(x)$ satisfies the assumption of Lemma \ref{estimate} such that 
$0<\gamma\leq 1$ and $L(x)$ is a slowly varying function (see \cite{F} for the definition). 
If $U$ is a bounded open subset of $(0,\infty)$ constructed from $f$ in Lemma \ref{estimate},  
then
$$ |(U+x)\ominus U|\asymp x^\gamma L(x).$$
\end{cor}

\medskip

For $0<\beta\leq 1/2$, we set $\varphi_\beta(x)=x^{\beta-1}e^{-x}$. 
Then $\varphi_\beta$ satisfies the assumption of Theorem \ref{typeIorIII}. 
Corollary \ref{symmdiff} shows that for any $0<\beta_1<\beta_2\leq 1/2$, there exists a bounded open set $U\subset (0,\infty)$ such that 
$\cA^{\varphi_{\beta_1}}(U)$ is a type III factor and $\cA^{\varphi_{\beta_2}}(U)$ is a type I factor. 
In particular $\alpha^{\varphi_{\beta_1}}$ and $\alpha^{\varphi_{\beta_2}}$ are not cocycle conjugate. 

\begin{theorem} There exist uncountably many mutually non-cocycle conjugate type III $E_0$-semigroups of the 
form $\alpha^\varphi$ with $\varphi\in L^1_{\mathrm{loc}}[0,\infty)\cap L^2((0,\infty),1\wedge xdx)_\R$.
\end{theorem}

\bigskip

%%%%%%%%%%%%%%%%%%%%%%%%%%%%%%%%%%%%%%%%%%%%%%%%%%%%%%%%%%%%%%%
%%%%%%%%%%%%%%%%%%%%%%%%%%%%%%%%%%%%%%%%%%%%%%%%%%%%%%%%%%%%%%%
%%%%%%%%%%%%%%%%%%%%%%%%%%%%%%%%%%%%%%%%%%%%%%%%%%%%%%%%%%%%%%%


\thebibliography{99}

\bibitem{Ara1} H. Araki,  
\textit{A lattice of von Neumann algebras associated with the quantum theory of a free Bose field.} 
J. Math. Phys. \textbf{4} (1963), 1343--1362. 

\bibitem{Ara2} H. Araki,  
\textit{Type of von Neumann algebra associated with free Bose field,}
Progr. Theoret. Phys. \textbf{32}, (1964), 956--965. 


\bibitem{Ara3} H. Araki, 
\textit{On quasifree states of the canonical commutation relations. II.} 
Publ. Res. Inst. Math. Sci. \textbf{7} (1971/72), 121--152. 


\bibitem{Ar} W. Arveson, 
\textit{Continuous analogues of Fock spaces}, Mem. Americ. Math. Soc. 80(409):1-66, 1989. 

%\bibitem{Arv2} W. Arveson, {\it Continuous analogues of Fock
%spaces IV  Essential states,} Acta Math. 164 (3/4) 265-300, 1990. 
  
\bibitem{Arv} W. Arveson, \textit{Non-commutative dynamics and $E$-semigroups}, Springer Monograph in Math. (Springer 2003).



\bibitem{pdct} B. V. Rajarama Bhat and R. Srinivasan, \textit{On product systems arising from sum systems} Infinite dimensional analysis and related topics, Vol. 8, Number 1, March 2005.


\bibitem{vD} A. van Daele,  
\textit{Quasi-equivalence of quasi-free states on the Weyl algebra.} 
Comm. Math. Phys. \textbf{21} (1971), 171--191.

\bibitem {F} W. Feller, 
\textit{An Introduction to Probability Theory and its Applications. Vol. II.} 
Second edition, John Wiley \& Sons, Inc., New York-London-Sydney 1971. 

\bibitem{H} K. Hoffman,  
\textit{Banach Spaces of Analytic Functions.} 
Prentice-Hall, Inc., Englewood Cliffs, N. J. 1962.

\bibitem{I} M. Izumi, 
\textit{A perturbation problem for the shift semigroup,} Preprint math/0702439, 2007.   

\bibitem{K} P. Koosis, 
\textit{Introduction to $H^p$ Spaces.} Second edition. With two appendices by V. P. Havin. 
Cambridge Tracts in Mathematics, 115. Cambridge University Press, Cambridge, 1998.


\bibitem{Vol} V. Liebscher, \textit{Random sets and invariants for
(type $II$) continuous product systems of Hilbert spaces},
Preprint math.PR/0306365.

\bibitem{KRP} K. R. Parthasarathy, \textit{An Introduction to Quantum
Stochastic Calculus}, Birkauser Basel, Boston, Berlin (1992). 

\bibitem {Po1} R. T. Powers, \textit{A nonspatial continuous semigroup of $*$-endomorphisms
of $B(H)$,} Publ. Res. Inst. Math. Sci.  23 (1987), 1053-1069.

\bibitem {Po2} R.T. Powers, {\it New examples of continuous spatial semigroups of endomorphisms of $B(H)$,} Int. J. Math. 10 (2):215-288, (1999).

\bibitem {Pr} G. L. Price, B. M. Baker, P. E. T. Jorgensen and P. S. Muhly
(Editors), {\it Advances in Quantum Dynamics}, 
(South Hadley, MA, 2002) Contemp.
Math. 335, Amer. Math. Society, Providence, RI (2003).

\bibitem{S} Y. Shalom, 
\textit{Harmonic analysis, cohomology, and the large-scale geometry of amenable groups.}
Acta Math. \textbf{192} (2004), 119--185.

\bibitem{T1} B. Tsirelson, 
\textit{Non-isomorphic product systems.} 
Advances in Quantum Dynamics (South Hadley, MA, 2002), 273--328, 
Contemp. Math., 335, Amer. Math. Soc., Providence, RI, 2003. 

\bibitem{T2} B. Tsirelson,  
\textit{Spectral densities describing off-white noises.} 
Ann. Inst. H. Poincar\'e Probab. Statist. \textbf{38} 
(2002), 1059--1069.


\bibitem{Y} K. Yosida,  
\textit{Functional Analysis.} 
Sixth edition. Springer-Verlag, Berlin-New York, 1980. 

\bibitem{Z} K. H. Zhu, 
\textit{Operator Theory in Function Spaces.} 
Monographs and Textbooks in Pure and Applied Mathematics, 139. Marcel Dekker, Inc., 
New York, 1990.
\end{document}